\documentclass{amsart}

\usepackage[utf8]{inputenc}
\usepackage[english]{babel}
\usepackage[margin=1.2in]{geometry}
\usepackage{amsmath, amsthm, amscd, amssymb, mathtools}
\usepackage{verbatim} 
\usepackage{wrapfig}
\usepackage{bbm}
\usepackage[nocompress]{cite}

\usepackage{hyperref}
\hypersetup{
  colorlinks   = true, 
  urlcolor     = blue, 
  linkcolor    = blue, 
  citecolor   = red 
}

\usepackage{tikz-cd}
\tikzset{
  symbol/.style={
    draw=none,
    every to/.append style={
      edge node={node [sloped, allow upside down, auto=false]{$#1$}}}
  }
}

\usepackage{enumitem}
\setenumerate{itemsep=0pt,topsep=3pt}
\setenumerate[1]{label=\textup{(\roman*)}}

\theoremstyle{plain}
\newtheorem{theorem}{Theorem}[section]
\newtheorem{lemma}[theorem]{Lemma}
\newtheorem{proposition}[theorem]{Proposition}
\newtheorem{corollary}[theorem]{Corollary}

\theoremstyle{definition}
\newtheorem{definition}[theorem]{Definition}

\newtheorem{example}[theorem]{Example}

\DeclareMathOperator{\Hom}{Hom}

\newcommand{\id}{\mathrm{id}}

\DeclareMathOperator{\rk}{rk}

\DeclareMathOperator{\maxspec}{maxspec}

\newcommand{\tensor}{\otimes}
\newcommand{\lla}{\langle\!\langle}
\newcommand{\rra}{\rangle\!\rangle}

\newcommand{\Z}{\ensuremath{\mathbb{Z}}}
\newcommand{\Q}{\ensuremath{\mathbb{Q}}}

\newcommand{\m}{\ensuremath{\mathfrak{m}}}
\newcommand{\p}{\ensuremath{\mathfrak{p}}}

\newcommand{\fm}{\ensuremath{\mathfrak{m}}}
\newcommand{\fn}{\ensuremath{\mathfrak{n}}}
\newcommand{\fp}{\ensuremath{\mathfrak{p}}}

\newcommand{\ii}{\textup{i}}

\title
[On canonical roots of fractional ideals]
{On canonical roots of fractional ideals}
\author[D.\ M.\ H. van Gent]{D.\ M.\ H. van Gent}
\address{Centrum Wiskunde \& Informatica, The Netherlands}
\email{dmhvg@cwi.nl}
\keywords{commutative algebra}

\begin{document}

\begin{abstract}
We give an algorithm to compute in polynomial time the roots of a fractional ideal of an order \(R\). 
We take care not to assume \(R\) is Dedekind, since the maximal order of a number field is generally inaccessible in polynomial time.
Consequently, the output of such an algorithm is no longer uniquely defined. For it to be a satisfying algorithm we additionally require it be functorial, i.e., isomorphisms on the inputs should induce isomorphisms on the outputs. To adhere to these two constraints, we generalize results from Dade--Taussky--Zassenhaus, and Ge and Buchmann--Eisenbrand.
\end{abstract}

\maketitle

\section{Introduction}

Let \(K\) be a number field. By factoring the polynomial \(X^n-a\) over \(K\), we may efficiently test if \(a\in K\) has an \(n\)-th root in \(K\) and if so compute it \cite{factoring_polynomials}.
It is a rule of thumb that any sufficiently multiplicative statement on the elements of \(K^*\) might be better stated for the invertible ideals of the maximal order of $K$ instead. 
Accordingly, Belabas--Simon \cite{Belabas} gave a polynomial-time algorithm that, given the maximal order of a number field and an invertible ideal \(\mathfrak{a}\) of this order, computes the roots of~\(\mathfrak{a}\).
The maximal order however is generally inaccessible in polynomial time; the difficulty of computing it is similar to that of prime factorization (cf.\ Theorem~1.3 in \cite{Buchman-Lenstra}).
This algorithmic inaccessibility can be mitigated by a design paradigm used in for example \cite{Buchman-Lenstra,Ge,Buchmann-Eisenbrand}.
Instead of formulating an algorithm for the maximal order, one allows an arbitrary order~\(R\).
Then, if a step of the algorithm fails because \(R\) does not resemble the maximal order well enough, the algorithm should instead return a strictly larger order of the same number field, with which we can replace \(R\).
Essentially, \(R\) is an approximation of the maximal order.
In this paper we present such a polynomial-time root-taking algorithm for fractional ideals of orders.
Unlike the algorithm of Belabas--Simon, ours does not require additional ramification data.

We encounter two ways an order \(R\) can fail to resemble the maximal order \(\mathcal{O}\).
Firstly, it can be that an ideal \(\mathfrak{a}\) of \(R\) has no (non-trivial) roots, but that it has roots in \(\mathcal{O}\) (Example~\ref{ex:need_extension}).
Secondly, even if a root exists in \(R\), it need not be unique or even characteristic, i.e., fixed by the automorphisms that fix \(\mathfrak{a}\) (Example~\ref{ex:non-functor}).
Uniqueness is an essential obstruction; if we could efficiently compute an order \(R\) where for each \(n\) each fractional ideal has a unique \(n\)-th root, then we could efficiently compute \(\mathcal{O}\) (Proposition~\ref{prop:torsion_free_compute_OK}).
Instead, we impose that our algorithms compute a functor (Section~\ref{sec:algorithmic_clich_e}).

We will prove our results for a broader class of rings than in the PhD thesis of the author \cite{thesis}, on which the present paper is based: for us, an \emph{order} will be a commutative ring $R$ which is free of finite rank as \(\Z\)-module, which consequently may contain zero-divisors, and a \emph{fractional ideal} of \(R\) is a finitely generated \(R\)-submodule \(\mathfrak{a}\subseteq \Q R\) such that \(\Q \mathfrak{a}=\Q R\).
To this end, we will similarly generalize results of Dade--Taussky--Zassenhaus \cite{Dade-Zassenhaus}, and Ge and Buchmann--Eisenbrand \cite{Ge,Buchmann-Eisenbrand}.
Although we opt to present full proofs for completeness, the contributions of the present paper to these generalizations should be considered minor.
Our main result is as follows.

\begin{theorem}[Informal]\label{thm:informal}
There is a functorial polynomial-time algorithm that, given an order $R$ and a fractional ideal $\mathfrak{a}$ of $R$, computes the maximal $n\in\Z_{\geq 0}$ with respect to divisibility such that there exist an order $R\subseteq S\subseteq \Q R$ and a fractional ideal $\mathfrak{b}$ of $S$ such that $S\mathfrak{a}=\mathfrak{b}^n$, and additionally computes such $S$ and $\mathfrak{b}$. 
\end{theorem}

In the edge case that \(\mathfrak{a}\) is \emph{torsion}, i.e. satisfies $\mathfrak{a}^\ell=R$ for some $\ell>0$, there is some order \(S\) such that \(S=S\mathfrak{a}\) (Lemma~\ref{lem:torsion_equiv}).
Then we have \(\mathfrak{b}^k=S\mathfrak{a}\) for \(\mathfrak{b}=S\) and all \(k\geq 0\), so the algorithm will return \(n=0\) on input $(R,\mathfrak{a})$.

Theorem~\ref{thm:informal} is informal in the sense that we have not defined what it means for the algorithm to be functorial. We will now restate it more precisely.
Consider the category $\textup{Frc}$ whose objects are pairs $(R,\mathfrak{a})$, where $R$ is an order and $\mathfrak{a}$ is a fractional ideal of $R$, and where the morphisms $(R,\mathfrak{a})\to(R',\mathfrak{a}')$ are the ring isomorphisms $f:\Q R\to \Q R'$ such that $f(R)=R'$ and $f(\mathfrak{a})=\mathfrak{a}'$. 
For $n\in\Z_{\geq 0}$, an \emph{$n$-th root} of $(R,\mathfrak{a})\in\textup{Frc}$ is an $(S,\mathfrak{b})\in\textup{Frc}$ such that $R\subseteq S\subseteq \Q R$ and $\mathfrak{b}^n=S\mathfrak{a}$.

\begin{theorem}[Theorem~\ref{thm:max_ideal_root}]\label{mainthm:max_ideal_root}
There exists a functor $F:\textup{Frc}\to\Z_{\geq0}\times \textup{Frc}$ such that for the objects $A\in\textup{Frc}$ and $(n,B)=F(A)$ we have that $B$ is an $n$-th root of $A$ such that $n$ is maximal with respect to divisibility among all roots of $A$, and for the morphisms $f:A\to A'$ we have $F(f)=\id \times f$. 
Moreover, there exists a polynomial-time algorithm that takes as input objects $A,A'\in\textup{Frc}$ and a morphism $f:A\to A'$, and computes $F(A)$, $F(A')$ and $F(f)$.
\end{theorem}

One recovers Theorem~\ref{thm:informal} by applying Theorem~\ref{mainthm:max_ideal_root} to the identity map of $(R,\mathfrak{a})$.

We will proceed to state the non-algorithmic counterpart to Theorem~\ref{thm:informal}.
Let \(R\) be a commutative ring.
We write \(\textup{Q}(R)\) for the total ring of fractions of \(R\) and we identify \(R\) with its image in \(\textup{Q}(R)\).
Let \(A,B\subseteq\textup{Q}(R)\) be \(R\)-submodules.
We write \(A+B\) (resp.\ \(A\cdot B\)) for the \(R\)-module generated by the elements \(a+b\) (resp.\ \(a\cdot b\)) for \(a\in A\) and \(b\in B\).
We recursively define \(A^0=R\) and \(A^{k+1}=A^k\cdot A\) and write \(A:B=\{x\in\textup{Q}(R)\,|\, x\cdot B\subseteq A\}\).
A \emph{fractional ideal} of \(R\) is a finitely generated \(R\)-submodule \(\mathfrak{a}\subseteq \textup{Q}(R)\) such that \(\mathfrak{a}\cdot \textup{Q}(R)=\textup{Q}(R)\), and an \emph{invertible ideal} of \(R\) is an \(R\)-submodule \(\mathfrak{a}\subseteq\textup{Q}(R)\) for which there exists an \(R\)-submodule \(\mathfrak{b}\subseteq\textup{Q}(R)\) such that \(\mathfrak{a}\cdot \mathfrak{b} = R\).

\begin{theorem}[Theorem~\ref{thm:root_exists_condition}]\label{mainthm:root_exists_condition}
Let \(Z\subseteq R\subseteq S \subseteq \textup{Q}(R)\) be commutative (sub)rings such that \(Z\) is Dedekind and \(S\) is finitely generated as a \(Z\)-module, let \(\mathfrak{a}\subseteq R\) be an invertible ideal and let \(m\in\Z_{\geq 0}\).
Write \(a=\mathfrak{a}\cap Z\) and suppose \(R/\mathfrak{a}\) is finite-\'etale over \(Z/a\).
If there exists an ideal \(\mathfrak{b}\subseteq S\) with \(S\mathfrak{a}=\mathfrak{b}^m\), then there exists an ideal \(b\subseteq Z\) with \(a=b^m\).
\end{theorem}

The algorithm of \cite{Belabas} and the algorithm of Theorem~\ref{thm:informal} can both be summarized as partially factoring $\mathfrak{a}$ depending on the $\Z/(\mathfrak{a}\cap\Z)$-algebra structure of $R/\mathfrak{a}$ until the maximal root of $\mathfrak{a}$ can be read off from this factorization.
In fact, for the inputs in which both algorithms apply, their behavior is essentially the same.
We use Theorem~\ref{mainthm:root_exists_condition} (cf.\ Lemma~2.2 in \cite{Belabas}) to show that once we are unable to further factor $\mathfrak{a}$ we have indeed computed its maximal root.
For the factorization of $\mathfrak{a}$ we require a generalization of the factor refinement algorithm of Ge \cite{Ge} for fractional ideals of orders in number rings.
Buchmann--Eisenbrand \cite{Buchmann-Eisenbrand} show that the algorithm of Ge can be made functorial using results from Dade--Taussky--Zassenhaus \cite{Dade-Zassenhaus} on blowups.
We will generalize these results to a broader class of rings.

For a maximal ideal \(\m\subset R\) we define \(n(\m)\) to be the minimal cardinality such that every ideal of the localization \(R_\m\) can be generated as $R_\m$-module by \(n(\m)\) elements.
We extend the definition of \(n\) to fractional ideals \(\mathfrak{a}\) of \(R\), where \(n(\mathfrak{a})=\max\{n(\m) : \mathfrak{a}_\m \neq R_\m\}\).
If \(R\) is finitely generated by say \(r\) elements as a module over some principal ideal domain \(Z\), for example \(Z=\Z\), then \(n(\mathfrak{a})\leq r\) for all fractional ideals \(\mathfrak{a}\) of \(R\).
The assumptions on \(R\) in the following theorem imply that \(n(\mathfrak{a})\) is always a finite quantity (Theorem~\ref{thm:finite_local_generators}).

\begin{theorem}[Theorem~\ref{thm:main_ideal}]\label{mainthm:main_ideal}
Let \(R\) be a Noetherian commutative ring of Krull dimension at most \(1\) such that \(\textup{Q}(R)\) is Artinian and let \(\mathfrak{a}\) be a fractional ideal of \(R\).
Among all subrings \(R\subseteq S \subseteq\textup{Q}(R)\) such that \(S\mathfrak{a}\) is an invertible ideal of \(S\), there exists a unique minimum with respect to inclusion.
This ring is Noetherian as \(R\)-module and equals \(\mathfrak{a}^{n}:\mathfrak{a}^{n}\) for all \(n+1\geq n(\mathfrak{a})\). 
\end{theorem}

Theorem~\ref{mainthm:main_ideal} generalizes Theorem~1.5.2 from \cite{Dade-Zassenhaus}, in the sense that \(R\) need not be a domain nor have a finitely generated integral closure.
We call the ring \(S\) from the theorem the \emph{blowup} of \(R\) at \(\mathfrak{a}\). 
One can show that it is a blowup in an algebro-geometric sense as defined in Proposition~7.14 in \cite{Hartshorne}. 
It is also not too difficult to prove most of Theorem~\ref{mainthm:main_ideal} and that the geometric blowup is affine by using established theory from algebraic geometry, as communicated to the author privately by Sebastian Bozlee.
However, it is unclear to the author whether the result that expresses \(S\) as \(\mathfrak{a}^{n}:\mathfrak{a}^{n}\) for bounded \(n\) can be similarly obtained.
From Theorem~\ref{mainthm:main_ideal} we may trivially construct an algorithm.

\begin{theorem}[Theorem~\ref{thm:blowup}]\label{mainthm:blowup}
There exists a polynomial-time algorithm that, given an order \(R\) and a fractional \(R\)-ideal \(\mathfrak{a}\), computes the unique minimal order \(R\subseteq S\subseteq \Q R\) such that \(S\mathfrak{a}\) is invertible.
\end{theorem}

We will proceed with the factor refinement algorithm.
For an order \(R\) and a set \(X\) of fractional ideals contained in \(R\), we define a \emph{coprime basis} of \(X\) to be a set \(Y\) of pairwise coprime invertible ideals contained in \(R\) such that \(X\) is contained in the group generated by \(Y\). 
For example, $\{5\Z, 6\Z\}$ is a coprime basis for $\{150\Z,180\Z\}$.
Note that a coprime basis need not exist; at the very least \(X\) must consist of invertible ideals. 
Instead, we allow factor refinement to enlarge the order similar to Theorem~\ref{thm:informal} and Theorem~\ref{mainthm:blowup}.
Aside from invertibility issues, we could take the coprime basis of $X$ consisting of all primes dividing an ideal of \(X\).
However, prime factorizations are generally inaccessible in polynomial time.
In a sense, a coprime basis approximates such a prime factorization. 

\begin{theorem}[Theorem~\ref{thm:coprime_basis}]
There exists a polynomial-time algorithm that, given an order \(R\) and a finite set \(X\) of fractional ideals contained in \(R\), computes the minimal order \(R\subseteq S\subseteq \Q R\) such that \(\{S\mathfrak{a} \,|\, \mathfrak{a}\in X \}\) has a coprime basis, and then computes the minimal coprime basis of \(\{S\mathfrak{a} \,|\, \mathfrak{a}\in X \}\).
\end{theorem}

The author would like to thank his PhD advisor Hendrik Lenstra. 
The author would also like to thank Sebastian Bozlee for their discussion on the geometric interpretation of the blowup, and two anonymous reviewers for their valuable comments.
No AI was used in the making of this paper.

\section{Fractional ideals}

Throughout this section \(R\) is a commutative ring. 
In this section we establish some preliminary results on fractional ideals of \(R\).
Most of these are established results that can be found in more specialized or generalized form in for example  Sections~2 and~11 of \cite{Eisenbud} and Section II.5 of \cite{Bourbaki}.
We write \(\textup{Q}(R)\) for the \emph{total ring of fractions} of \(R\), i.e., the localization of \(R\) obtained by adjoining inverses of its regular elements, and \(\maxspec R\) for the set of maximal ideals of \(R\).
We say \(R\) is \emph{semi-local} if \(\maxspec R\) is finite.

\begin{definition}
Let \(S\) be a commutative \(R\)-algebra and \(A,B\subseteq S\) be \(R\)-submodules. 
We write \(A+B\) (resp.\ \(A\cdot B\)) for the \(R\)-module generated by the elements \(a+b\) (resp.\ \(a\cdot b\)) for \(a\in A\) and \(b\in B\).
We recursively define \(A^0=R\) and \(A^{n+1}=A^n\cdot A\) for \(n\in\Z_{\geq 0}\).
We write \((A:B)_S=\{x\in S : xB\subseteq A\}\) or simply \(A:B\) for \((A:B)_{\textup{Q}(R)}\).
\end{definition}

\begin{definition}
A \emph{fractional ideal} of \(R\) is a finitely generated \(R\)-submodule \(\mathfrak{a}\subseteq \textup{Q}(R)\) such that \(\textup{Q}(R)\cdot \mathfrak{a}=\textup{Q}(R)\). 
An \emph{invertible ideal} of \(R\) is an \(R\)-submodule \(\mathfrak{a}\subseteq\textup{Q}(R)\) for which there exists an \(R\)-submodule \(\mathfrak{b}\subseteq\textup{Q}(R)\) such that \(\mathfrak{a}\cdot \mathfrak{b}=R\). 
We write \(\mathcal{I}(R)\) for the group of invertible ideals of \(R\).
We say a fractional ideal $\mathfrak{a}$ of $R$ is \emph{torsion} if $\mathfrak{a}^\ell=R$ for some $\ell>0$, or equivalently if $\mathfrak{a}$ is in the torsion subgroup of $\mathcal{I}(R)$.
\end{definition}

\begin{lemma}\label{lem:inv_basic}
Each invertible ideal \(\mathfrak{a}\) of \(R\) is a fractional ideal such that \(\mathfrak{a}\cdot(R:\mathfrak{a})=R\) and \(\mathfrak{a}:\mathfrak{a}=R\).
\end{lemma}
\begin{proof}
Suppose \(\mathfrak{a}\) and \(\mathfrak{b}\) are invertible ideals such that \(\mathfrak{a}\cdot \mathfrak{b}=R\).
We may write \(\sum_{k=1}^n a_k b_k=1\) for some \(n\in\Z_{\geq 0}\), \(a_1,\dotsc,a_n\in \mathfrak{a}\) and \(b_1,\dotsc,b_n\in \mathfrak{b}\). Let \(\mathfrak{a}'=\sum_{k=1}^n R a_k\) and note that \(1\in \mathfrak{a}'\cdot \mathfrak{b}\subseteq R\), so \(\mathfrak{a}\cdot \mathfrak{b}=R=\mathfrak{a}'\cdot \mathfrak{b}\), and \(\mathfrak{a}=\mathfrak{a}'\) is finitely generated. 
From \(1\in \mathfrak{b}\cdot \mathfrak{a}\subseteq \textup{Q}(R)\cdot \mathfrak{a}\), we similarly conclude that \(\textup{Q}(R)\cdot \mathfrak{a} = \textup{Q}(R)\). 
By definition we have \(\mathfrak{b}\subseteq R:\mathfrak{a}\), so \(R=\mathfrak{a}\cdot \mathfrak{b}\subseteq \mathfrak{a}\cdot (R:\mathfrak{a})\subseteq R\) and \(\mathfrak{a}\cdot (R:\mathfrak{a})=R\).
Similarly \(R\subseteq \mathfrak{a}:\mathfrak{a} \subseteq \mathfrak{ab}:\mathfrak{ab} = R:R=R\) and \(\mathfrak{a}:\mathfrak{a}=R\).
\end{proof}

\begin{lemma}\label{lem:ideal_as_fraction}
Let \(\mathfrak{a}\) be a fractional \(R\)-ideal such that \(\mathfrak{a}\) and \(R+\mathfrak{a}\) are invertible.
Then \(R+\mathfrak{a}^{-1}\) is invertible and \(\mathfrak{b}=(R+\mathfrak{a}^{-1})^{-1}\) and \(\mathfrak{c}=(R+\mathfrak{a})^{-1}\) satisfy \textup{(1)} \(\mathfrak{b},\mathfrak{c}\subseteq R\); \textup{(2)} \(\mathfrak{b}+\mathfrak{c}=R\) and \textup{(3)} \(\mathfrak{a}=\mathfrak{b}:\mathfrak{c}\).
\end{lemma}
\begin{proof}
Note that \(R+\mathfrak{a}^{-1}=\mathfrak{a}^{-1}(R+\mathfrak{a})\) is invertible.
Rearranging using Lemma~\ref{lem:inv_basic} gives \(\mathfrak{a}=(R+\mathfrak{a}):(R+\mathfrak{a}^{-1})=\mathfrak{b}:\mathfrak{c}\).
We have \(\mathfrak{b}+\mathfrak{c}=\mathfrak{a}\mathfrak{c}+\mathfrak{c}=(\mathfrak{a}+R)\mathfrak{c}=R\).
In particular, we have \(\mathfrak{b},\mathfrak{c}\subseteq R\).
\end{proof}

\begin{proposition}\label{prop:semi_loc_inv}
Suppose \(R\) is semi-local and \(A\subseteq\textup{Q}(R)\) is an \(R\)-submodule. Then \(A\) is an invertible ideal if and only if there exists some \(a\in\textup{Q}(R)^*\) such that \(A=Ra\).
\end{proposition}
\begin{proof}
The if-part is trivial, so suppose \(\mathfrak{a}\) is an invertible ideal with \(\mathfrak{ab}=R\). 
For each \(\fm\in\maxspec R\) we may choose \(a_\fm\in \mathfrak{a}\) and \(b_\fm\in \mathfrak{b}\) such that \(a_\fm b_\fm \in R\setminus \fm\), and by prime avoidance (Lemma~3.3 in \cite{Eisenbud}) choose \(\lambda_\fm\in R\) such that for all \(\fn\in\maxspec R\) we have \(\lambda_\fm \in \fn\) if and only if \(\fm\neq\fn\).
Consider \(a=\sum_\fm \lambda_\fm a_\fm\) and \(b=\sum_\fm \lambda_\fm b_\fm\). Then \(ab=\sum_{\fm,\fn} \lambda_\fm \lambda_{\fn} a_\fm b_{\fn} \in R\).
For all \(\fm\) there is precisely one term not contained in \(\fm\), namely \(\lambda_\fm \lambda_\fm a_\fm b_{\fm}\), hence \(ab\not\in \fm\).
It follows that \(ab\in R^*\), and thus \(a\in\textup{Q}(R)^*\). 
Finally, \(aR \subseteq \mathfrak{a} = ab \mathfrak{a} \subseteq a \mathfrak{ba} = aR\) and \(\mathfrak{a}=Ra\).
\end{proof}

\begin{lemma}\label{lem:semi-local}
Let \(R\subseteq S\) be commutative (sub)rings so that \(S\) is Noetherian as \(R\)-module.
If \(\fm\in\maxspec S\), then \(\fm\cap R\in\maxspec R\). 
If \(R\) is semi-local, then \(S\) is semi-local.
\end{lemma}
\begin{proof}
The first statement is Corollary~4.17 in \cite{Eisenbud}.
Suppose \(R\) is semi-local. 
If \(\fp\in\maxspec R\), then \(S/\fp S\) is a finite dimensional \(R/\fp\)-vector space. 
Hence \(S/\fp S\) has only finitely many (maximal) ideals.
Each maximal ideal of \(S\) lies over a maximal ideal of \(R\) by the first statement, of which there are only finitely many, so \(S\) is semi-local.
\end{proof}

\begin{lemma}\label{lem:tensor_hom}
Suppose \(A,B\subseteq\textup{Q}(R)\) are \(R\)-submodules such that \(\textup{Q}(R)\cdot A=\textup{Q}(R)\cdot B=\textup{Q}(R)\).
Then the natural maps \(A\tensor_R B\to A\cdot B\) and \(A:B\to\Hom_R(B,A)\) are isomorphisms.
\end{lemma}
\begin{proof}
By flatness of \(\textup{Q}(R)\) the natural map \(\textup{Q}(R)\tensor D\to\textup{Q}(R)\) is an isomorphism for all \(R\)-submodules \(D\subseteq\textup{Q}(R)\) such that \(\textup{Q}(R)\cdot D = \textup{Q}(R)\). 
Any exact sequence \(0\to C\to A\tensor B \to A\cdot B \to 0\) where the third map is multiplication induces an exact sequence \(0\to\textup{Q}(R)\tensor C \to \textup{Q}(R) \to \textup{Q}(R)\to 0\) where the third map is the identity.
Hence \(\textup{Q}(R)\tensor C=0\). It follows that \(C=0\) and thus \(A\tensor B\to A \cdot B\) is an isomorphism.

As \(\textup{Q}(R)\cdot B=\textup{Q}(R)\), we may choose some \(b\in B\cap \textup{Q}(R)^*\). The inverse to \(A:B\to\Hom_R(B,A)\) is then given by \(f\mapsto f(b)/b\).
\end{proof}

For an \(R\)-module \(M\) and prime ideal \(\fp\subset R\) we write \(M_\fp\) for the localization of \(M\) at \(\fp\).

\begin{lemma}\label{lem:localize_fractional_ideal}
Let \(\fp\subset R\) be a prime ideal. We have a natural morphism \(\textup{Q}(R)\to\textup{Q}(R_\fp)\) of \(R\)-algebras that induces an isomorphism \(\textup{Q}(R)_\fp \to \textup{Q}(R_\fp)\) of \(R_\fp\)-algebras.
\end{lemma}
\begin{proof}
Localization preserves injectivity, hence the image of each regular element of \(R\) in \(R_\p\) is regular.
Consequently, the natural map \(\textup{Q}(R)\to\textup{Q}(R_\fp)\) is well-defined. One verifies that the natural map \(\textup{Q}(R)_\fp\to\textup{Q}(R_\fp)\) given by \((a/c)/b\mapsto (a/b)/(c/1)\) has inverse \((a/b)/(c/d)\mapsto (ad/c)/b\).
\end{proof}

By Lemma~\ref{lem:localize_fractional_ideal} we may for each \(R\)-submodule \(A\subseteq\textup{Q}(R)\) interpret \(A_\fp\) as an \(R_\fp\)-submodule of \(\textup{Q}(R_\fp)\).
From the following lemma we conclude that localization maps fractional ideals to fractional ideals and invertible ideals to invertible ideals.

\begin{lemma}\label{lem:arith_loc}
Suppose \(A,B\subseteq\textup{Q}(R)\) are \(R\)-submodules and \(\fp\subset R\) is prime. Then \((A+B)_\fp=A_\fp+B_\fp\) and \((A\cdot B)_\fp=A_\fp\cdot B_\fp\) hold. If \(A\) and \(B\) are finitely presented, then \((A:B)_\fp=A_\fp:B_\fp\) holds.
\end{lemma}
\begin{proof}
The only non-trivial assertion is the last, where we combine Lemma~\ref{lem:tensor_hom} with the fact that \(\Hom\) commutes with localization when \(A\) and \(B\) are finitely presented (Proposition~2.10 in \cite{Eisenbud}).
\end{proof}

\begin{lemma}\label{lem:inv_loc_prop}
Suppose \(A\subseteq\textup{Q}(R)\) is a finitely presented \(R\)-submodule. Then \(A\) is an invertible ideal of \(R\) if and only if \(A_\fm\) is an invertible ideal of \(R_\m\) for each \(\fm\in\maxspec R\).
\end{lemma}
\begin{proof}
The only-if-part follows from Lemma~\ref{lem:arith_loc}, so suppose \(A_\fm\) is an invertible ideal for each \(\fm\in\maxspec R\).
We have \((A\cdot (R:A))_\fm=A_\fm\cdot(R_\fm:A_\fm)=A_\fm\) for each \(\fm\in\maxspec R\) by Lemma~\ref{lem:inv_basic} and Lemma~\ref{lem:arith_loc}, so the natural map \(A\cdot (R:A)\to R\) is locally, hence globally, an isomorphism. 
Thus \(A\) is an invertible ideal.
\end{proof}

Lemma~\ref{lem:inv_loc_prop} shows that, for Noetherian \(R\), invertibility is a local property. We will show in Proposition~\ref{prop:invertibility_defined_locally} that under mild hypotheses the local components are completely independent.

\begin{definition}
Let \(M\) be an \(R\)-module.
We write \(\ell_R(M)\) for the \emph{length} of \(M\) as \(R\)-module. 
If \(M\) has finite length we write \([M]_R\) for the map \(\maxspec R\to \Z_{\geq0}\) given by \(\mathfrak{m}\mapsto \ell_{R_\mathfrak{m}}(M_\mathfrak{m})\).
\end{definition}

\begin{lemma}[Theorem 2.13 in \cite{Eisenbud}]\label{lem:local_length}
Let \(R\) be a commutative ring and \(N\subseteq M\) be \(R\)-modules.
Then \(\ell_R(M)=\ell_R(N)+\ell_R(M/N)\) and \(\ell_R(M)=\sum_{\mathfrak{m}\in \maxspec R} \ell_{R_\mathfrak{m}}(M_\mathfrak{m})\). \qed
\end{lemma}

\begin{lemma}\label{lem:artinian_quotient}
Suppose \(R\) is Noetherian of Krull dimension at most \(1\). If \(\mathfrak{a}\subseteq R\) is a fractional ideal, then \(R/\mathfrak{a}\) has finite length.
\end{lemma}
\begin{proof}
Let \(\overline{\p}_1\subset\dotsm\subset\overline{\p}_n\) be a chain of prime ideals of \(R/\mathfrak{a}\) of maximal length.
Its pullback in \(R\) is a chain of prime ideals \(\p_1\subset\dotsm\subset\p_n\).
From $Q(R)\cdot\mathfrak{a}=Q(R)$ it follows that $\mathfrak{a}$, and hence $\p_1$, contains a regular element of $R$.
Hence \(\p_1\) is not a minimal prime (Theorem~3.1 in \cite{Eisenbud}).
We conclude that the Krull dimension of \(R/\mathfrak{a}\) is strictly less than that of \(R\).
Hence \(R/\mathfrak{a}\) has finite length (Corollary 9.1 in \cite{Eisenbud}).
\end{proof}

\begin{lemma}\label{lem:quotient_length}
Let \(\mathfrak{a}\subseteq\mathfrak{b}\) be fractional ideals of \(R\).
Then there exists some regular \(r\in R\) such that \(r\mathfrak{b}\subseteq\mathfrak{a}\).
If \(R\) is Noetherian of Krull dimension at most \(1\), then \(\ell_R(\mathfrak{b}/\mathfrak{a})<\infty\).
\end{lemma}
\begin{proof}
We may choose generators \(r_1/s_1,\dotsc,r_n/s_n\in\textup{Q}(R)\) of \(\mathfrak{b}\).
Then \(y=\prod_i s_i\) is regular and satisfies \(y\mathfrak{b}\subseteq R\). 
Since \(\mathfrak{a}\) is fractional there exists some regular \(x\in R\cap\mathfrak{a}\).
Hence we may take \(r=xy\).
We have a surjection \(\mathfrak{b}/xR\to\mathfrak{b}/\mathfrak{a}\), so it suffices to show \(\ell_R(\mathfrak{b}/xR)<\infty\).
The injection \(\mathfrak{b}/xR\to \frac{1}{y} R / x R\) and the fact that \(\tfrac{1}{y}R/xR\cong R/xyR\) has finite length by Lemma~\ref{lem:artinian_quotient} finish the proof.
\end{proof}

\begin{proposition}\label{prop:invertibility_defined_locally}
Suppose \(R\) is Noetherian of dimension at most \(1\) and \(\textup{Q}(R)\) is Artinian. 
Then we have a group isomorphism
\[ \mathcal{I}(R)\to\bigoplus_{\m\in\maxspec R} \mathcal{I}(R_\mathfrak{m}) \]
that sends a fractional ideal to its vector of localizations. 
\end{proposition}
\begin{proof}
By Lemma~\ref{lem:local_length} and Lemma~\ref{lem:artinian_quotient} each \(\mathfrak{a}\in\mathcal{I}(R)\) satisfies \(\mathfrak{a}_\mathfrak{m}=R_\mathfrak{m}\) for almost all \(\m\in\maxspec R\), so the map is well-defined.
It is injective because equality of submodules is a local property.
To see it is surjective, let \(\fm\in\maxspec R\) and let \(\mathfrak{a}\subset R_\fm\) be a fractional ideal. Let \(\tilde{\mathfrak{a}}\) be the pullback of \(\mathfrak{a}\) along \(\textup{Q}(R)\to\textup{Q}(R_\fm)\) and let \(\mathfrak{b}=\tilde{\mathfrak{a}}\cap R\). Clearly \(\mathfrak{b}_\fm=\tilde{\mathfrak{a}}_\fm \cap R_\fm = \mathfrak{a} \cap R_\fm = \mathfrak{a}\). 
Since \(\mathfrak{a}\) is fractional by Lemma~\ref{lem:inv_basic}, the ring \(R_\fm/\mathfrak{a}\) is Artinian by Lemma~\ref{lem:artinian_quotient}, and thus its unique prime ideal is nilpotent.
Hence we may choose some \(k\) such that \(\fm_\fm^k\subseteq\mathfrak{a}\).
Let \(\fn\in\maxspec R\) be distinct from \(\fm\). 
Choose some \(x\in \fm^k\setminus\fn\), which exists since \(\fn\) is prime and \(\fm\) is maximal.
Then \(x\in \mathfrak{b}\), yet \(x/1\in \mathfrak{b}_\fn\) is a unit. Hence \(\mathfrak{b}_\fn=R_\fn\).
From Lemma~\ref{lem:inv_loc_prop} we conclude that \(\mathfrak{b}\) is an invertible ideal with image \(\mathfrak{a}\) in \(\bigoplus_{\m} \mathcal{I}(R_\m)\). Hence the map is surjective.
\end{proof}

\begin{corollary}\label{cor:invertible_descent}
Suppose \(R\) is Noetherian of Krull dimension at most \(1\) and \(\textup{Q}(R)\) is Artinian. Let \(R\subseteq S\subseteq \textup{Q}(R)\) be (sub)rings so that \(S\) is Noetherian as \(R\)-module. Then for every \(\mathfrak{b}\in\mathcal{I}(S)\) there exists an \(\mathfrak{a}\in\mathcal{I}(R)\) such that \(S\mathfrak{a}=\mathfrak{b}\), and such that \(\mathfrak{a}_\m=R_\m\) for all \(\m\in\maxspec R\) satisfying \(\mathfrak{b}_\m=S_\m\). 
\end{corollary}
\begin{proof}
By Proposition~\ref{prop:invertibility_defined_locally} we may construct \(\mathfrak{a}\) locally, and thus assume \(R\) is a local ring with maximal ideal \(\fm\). 
If $\mathfrak{b}=S$, then we choose $\mathfrak{a}=R$ to ensure the final condition on $\mathfrak{a}$ holds.
Otherwise, since the ring \(S\) is semi-local by Lemma~\ref{lem:semi-local}, and hence \(\mathfrak{b}=Sb\) for some \(b\in\textup{Q}(R)\) by Proposition~\ref{prop:semi_loc_inv}, we may simply use \(\mathfrak{a}=Rb\).
\end{proof}

\begin{proposition}\label{prop:invertible_ideal_finite_extension}
Let \(R\subseteq S\subseteq\textup{Q}(R)\) be commutative (sub)rings. If \(R\) is Noetherian of Krull dimension at most~\(1\) and \(\ell_R(S/R)<\infty\), then for all invertible ideals \(\mathfrak{a}\subseteq R\) we have \([S/S\mathfrak{a}]_R=[R/\mathfrak{a}]_R\).
\end{proposition}
\begin{proof}
It suffices to prove for local \(R\) that \(\ell(S/S\mathfrak{a})=\ell(R/\mathfrak{a})\).
In this case \(\mathfrak{a}=Ra\) for some regular \(a\in R\) by Proposition~\ref{prop:semi_loc_inv}.
Note that the map \(S/R \to S\mathfrak{a}/ \mathfrak{a}\) given by \(x\mapsto a x\) is an isomorphism since \(a\) is regular in \(S\).
We conclude that \(\ell(S\mathfrak{a}/\mathfrak{a})=\ell(S/R)<\infty\).
We have exact sequences of \(R\)-modules
\[0\to R/\mathfrak{a} \to S/\mathfrak{a} \to S/R\to 0 \quad\text{and}\quad 0\to S\mathfrak{a}/\mathfrak{a}\to S/\mathfrak{a}\to S/S\mathfrak{a}\to 0.\]
Hence \(\ell(R/\mathfrak{a})+\ell(S/R)=\ell(S\mathfrak{a}/\mathfrak{a})+\ell(S/S\mathfrak{a})\) by Lemma~\ref{lem:local_length} and Lemma~\ref{lem:artinian_quotient}.
\end{proof}

\begin{example}
With the notation as in Proposition~\ref{prop:invertible_ideal_finite_extension}, the \(R\)-modules \(R/\mathfrak{a}\) and \(S/S\mathfrak{a}\) need not be  isomorphic. 
In fact, if \(R\) and \(S\) are orders, then the modules need not even be isomorphic over \(\Z\).

Take \(R=\Z[2\ii]\) and \(S=\Z[\ii]\) with \(\mathfrak{a}=2\ii R\), which is clearly invertible.
Then \(R/\mathfrak{a}\cong \Z/4\Z\) and \(S/S\mathfrak{a} \cong (\Z/2\Z)^2\) as \(\Z\)-modules, which are non-isomorphic.
\end{example}

\section{Algorithmic preliminaries}\label{sec:algorithmic_clich_e}

We briefly describe how we choose to represent our mathematical objects algorithmically.
Ideally, whether an algorithm runs in polynomial time is independent of this choice, as long as the choice is reasonable. Hence the coming sections are written in a sufficiently abstract way that is agnostic to the choice of representation, following the style of \cite{Lenstra_algorithms}.

We encode a finitely generated free Abelian group by an integer \(n\), its rank, and a linear map between such groups as a matrix in the natural way. A subgroup is represented by (the image of) an injective linear map.
Using the LLL-algorithm \cite[Section~5]{Lattices} we may replace the representation of a full-rank subgroup by an equivalent one, i.e., one with the same image, such that the lengths of the coefficients of the matrix representing this map are bounded by a polynomial in the length of the index of this subgroup.
Consequently, iterative computations avoid coefficient blowup as long as we properly bound the index.

Recall that an order is a commutative ring that as an Abelian group is finitely generated and free.
We encode an order as a finitely generated free Abelian group together with some \(M\in\Z^{n\times n\times n}\) describing a multiplicative structure on \(\Z^n\):
\[ (x_1,\dotsc,x_n) \cdot (y_1,\dotsc,y_n) = \Big( \sum_{i,j} x_i y_j M_{i,j,k} \Big)_k  \]
A fractional ideal of an order \(R\) is up to integer scaling a finitely generated subgroup of \(R\) and will be encoded as such.
For fractional ideals \(\mathfrak{a}\) and \(\mathfrak{b}\) of \(R\) we may efficiently compute \(\mathfrak{a}+\mathfrak{b}\), \(\mathfrak{a}\mathfrak{b}\) and \(\mathfrak{a}:\mathfrak{b}\).
Consequently, for an additional fixed constant $k\in\Z_{\geq 1}$ we may compute \(\mathfrak{a}^k\) in polynomial time.
However, should $k$ be variable, then $k$ must be encoded in unary if we desire a polynomial-time algorithm, as would be the case for integer exponentiation.

A concept that we will make use of in this paper is that of a \emph{functorial} algorithm, or equivalently an algorithm that computes a functor. 
Suppose $F:\mathcal{C}\to\mathcal{D}$ is a functor and suppose we have chosen a representation of the objects and morphisms of both $\mathcal{C}$ and $\mathcal{D}$.
We say an algorithm computes $F$ if the algorithm takes as input objects $X,Y\in\mathcal{C}$ and a morphism $f:X\to Y$, and outputs $F(X)$, $F(Y)$ and $F(f)$.
When dealing with category theory informally it is common to be vague about the action of a functor on the morphisms, as it is generally `obvious' what it should be.
We will treat morphisms similarly in the algorithms.

\section{Krull--Akizuki}

In this section we will prove a variation on the Krull--Akizuki theorem.
Throughout this section \(R\) is a commutative ring.
We write \(\textup{nil}(R)=\{x\in R\,|\, (\exists\, n\in\Z_{>0})\ x^n=0\}\) for the \emph{nilradical} of \(R\).
We call the ring 
\[ \{ x\in Q(R) \,|\, \exists\, \textup{monic } f \in R[X] \textup{ s.t. } f(x)=0 \}. \]
the \emph{integral closure} of $R$, and say $R$ is \emph{integrally closed} if it equals its integral closure.
Note that the integral closure of $R$ need not be a domain.

\begin{lemma}\label{lem:nilQ}
We have \(\textup{Q}(R)\cdot \textup{nil}(R)=\textup{nil}(\textup{Q}(R))\), 
and there is a natural \(R\)-algebra isomorphism \(\textup{Q}(R)/\textup{nil}(\textup{Q}(R))\to \textup{Q}(R/\textup{nil}(R))\).
\end{lemma}
\begin{proof}
Clearly \(\textup{Q}(R) \cdot \textup{nil}(R)\subseteq \textup{nil}(\textup{Q}(R))\). Conversely, suppose \(r/s\in\textup{nil}(\textup{Q}(R))\) with \(r,s\in R\) and \(s\) regular. Then \(0=(r/s)^n=r^n/s^n\) for some \(n\in\Z_{\geq 1}\) and thus \(r^n=0\). We conclude that \(r\in \textup{nil}(R)\), so \(r/s=(1/s) \cdot r \in \textup{Q}(R) \cdot \textup{nil}(R)\). Hence \(\textup{Q}(R) \cdot \textup{nil}(R)= \textup{nil}(\textup{Q}(R))\).

Suppose \(r\in R\) is regular. If \(s\in R\) satisfies \(rs\in\textup{nil}(R)\), then \(s\in\textup{nil}(R)\) by regularity of \(r\). Hence \(r\) is regular in \(R/\textup{nil}(R)\). Thus we have a well-defined surjection \(\textup{Q}(R)\to\textup{Q}(R/\textup{nil}(R))\). Clearly \(\textup{Q}(R)\cdot \textup{nil}(R)\) is contained in the kernel. Conversely, if \(r/s\) is in the kernel, then \(r\in\textup{nil}(R)\) and thus \(r/s\in \textup{nil}(\textup{Q}(R))\). Thus \(\textup{Q}(R)/\textup{nil}(\textup{Q}(R))\cong\textup{Q}(R/\textup{nil}(R))\).
\end{proof}

\begin{lemma}\label{lem:integral-closure-iff}
\(R\) is integrally closed if and only if \(\textup{nil}(R)=\textup{nil}(Q(R))\) and \(R/\textup{nil}(R)\) is integrally closed.
\end{lemma}
\begin{proof}
Write \(\mathfrak{n}=\textup{nil}(R)\).
Suppose \(R\) is integrally closed.
Each \(y\in \textup{nil}(\textup{Q}(R))\) is a root of \(X^n\in R[X]\) for some \(n\geq 1\), so \(y\in R\) and \(\textup{nil}(\textup{Q}(R))=\mathfrak{n}\).
Suppose \(x\in\textup{Q}(R/\mathfrak{n})\) and monic \(g\in (R/\mathfrak{n})[X]\) are such that \(g(x)=0\).
Then \(g\) lifts to some monic \(\tilde{g}\in R[X]\), and \(x\) lifts to some \(\tilde{x}\in\textup{Q}(R)\) along the map from Lemma~\ref{lem:nilQ}.
It follows that \(\tilde{x}\) is a root of \(\tilde{g}^n\) for some \(n\geq 1\). 
Hence \(\tilde{x}\) is integral over \(R\) and \(\tilde{x}\in R\). Thus \(x\in R/\mathfrak{n}\) and \(R/\mathfrak{n}\) is integrally closed.

Suppose that \(\mathfrak{n}=\textup{nil}(Q(R))\) and that \(R/\mathfrak{n}\) is integrally closed.
Let \(x\in\textup{Q}(R)\) be integral over \(R\). Then its image in \(\textup{Q}(R/\mathfrak{n})\) is integral over \(R/\mathfrak{n}\) and thus contained in \(R/\mathfrak{n}\).
The inverse image of \(R/\mathfrak{n}\) under \(\textup{Q}(R)\to\textup{Q}(R/\mathfrak{n})\) is \(R+\mathfrak{n}=R\), so \(x\in R\). Hence \(R\) is integrally closed.
\end{proof}

\begin{theorem}\label{thm:exist-invertible-extension}
Suppose \(R\) is Noetherian of Krull dimension at most \(1\) and \(\textup{Q}(R)\) is Artinian, and let \(\mathfrak{a}\) be a fractional ideal. 
Then there exists a subring \(R\subseteq S \subseteq\textup{Q}(R)\) which is Noetherian as \(R\)-module so that \(S\mathfrak{a}\) is invertible.
\end{theorem}
\begin{proof}
The ring \(\textup{Q}(R)\) is Noetherian and Artinian and thus a finite product of local rings. 
Each of its idempotents is integral over \(R\), and without loss of generality we may assume they are contained in \(R\).
Hence \(R\) factors in a way compatible with the factorization of \(\textup{Q}(R)\).
It suffices to prove the theorem for each factor, and thus without loss of generality we assume that \(\textup{Q}(R)\) is local with nilpotent maximal ideal~\(\mathfrak{n}\).

Let \(\mathcal{O}\) be the integral closure of \(R\). 
From Lemma~\ref{lem:integral-closure-iff} it follows both that \(\mathfrak{n}=\textup{nil}(\mathcal{O})\) and that \(\mathcal{O}/\mathfrak{n}\) is integrally closed.
As every zero divisor of $\mathcal{O}$ is a zero divisor of $\textup{Q}(R)$ and hence contained in $\mathfrak{n}$, we conclude that $\mathcal{O}/\mathfrak{n}$ is a domain.
Moreover, the Krull dimension of \(\mathcal{O}/\mathfrak{n}\) is at most that of \(\mathcal{O}\), and the latter is at most that of \(R\) by integrality of \(\mathcal{O}\) over \(R\) (Corollary~4.18 in \cite{Eisenbud}).
Hence by the Krull--Akizuki theorem (Corollary~VII.2.2 in \cite{Bourbaki}) the ring \(\mathcal{O}/\mathfrak{n}\) is Noetherian.
We conclude that \(\mathcal{O}/\mathfrak{n}\) is a Dedekind domain.
Hence \(\mathfrak{a}\cdot(\mathcal{O}/\mathfrak{n})\) is invertible (Theorem~VII.2.1 in \cite{Bourbaki}). 
Let \(\mathfrak{b}\) be the pullback of its inverse along \(\textup{Q}(\mathcal{O})\to\textup{Q}(\mathcal{O}/\mathfrak{n})\).
Then \(1+n\in \mathfrak{a}\mathfrak{b}\) for some \(n\in\textup{nil}(\textup{Q}(\mathcal{O}))\), but since \(1+n\in\mathcal{O}^*\) we have \(1\in\mathfrak{a}\mathfrak{b}\).
Moreover, \(\mathfrak{a}\mathfrak{b}\subseteq \mathcal{O}+\textup{nil}(\textup{Q}(\mathcal{O}))=\mathcal{O}\) by Lemma~\ref{lem:integral-closure-iff}, so \(\mathfrak{b}\) is the inverse of \(\mathcal{O}\mathfrak{a}\).

Note however that \(\mathcal{O}\) need not be a Noetherian \(R\)-module.
Instead write \(1=\sum_{i\in I}a_ib_i\) for some finite set \(I\) and \(a_i\in\mathfrak{a}\) and \(b_i\in\mathfrak{b}\).
Then \(S=R[a_ib_j : i,j\in I]\) is Noetherian over \(R\), and \(\mathfrak{b}'=\sum_{i\in I} S b_i\) satisfies \(\mathfrak{a}\mathfrak{b}'\subseteq S\) and \(1\in \mathfrak{a}\mathfrak{b}'\).
Hence \(\mathfrak{b}'\) is an inverse of \(S\mathfrak{a}\).
\end{proof}

\section{Blowing up}\label{sec:blowing_up}

In this section we generalize a result of Dade--Taussky--Zassenhaus (Theorem~1.5.2 in \cite{Dade-Zassenhaus}).
Paraphrased, their theorem states that for a fractional ideal $\mathfrak{a}$ in a one-dimensional Noetherian domain $R$ the ring
\[ \bigcup_{n\geq 1} (\mathfrak{a}^n : \mathfrak{a}^n),  \]
which we refer to as the \emph{blowup of $R$ at $\mathfrak{a}$}, is the smallest ring $R\subseteq S\subseteq Q(R)$ with respect to inclusion for which $S\mathfrak{a}$ is invertible.
Our generalization (Theorem~\ref{thm:main_ideal}) follows notes on lectures by Lenstra \cite{lecture} and is essentially the same, save for two changes. 
Firstly, we replace the maximal order in the proof of \cite{lecture} by the ring obtained by Theorem~\ref{thm:exist-invertible-extension} to allow $R$ to contain zero-divisors.
Secondly, we use a theorem of Sally \cite{generators} to effectively bound the point at which the limit defining the blowup stabilizes in the absence of an underlying Dedekind domain like $\Z$ over which $R$ is finitely generated (cf.\ Lemma~\ref{lem:iota_bound_basis}).

\begin{definition}
Let \(R\) be a commutative ring. For an \(R\)-module \(M\) write \(\gamma_R(M)\) for the minimal cardinality of a generating set of \(M\) as \(R\)-module.
For a fractional ideal \(\mathfrak{a}\) of \(R\) we define 
\[n(\mathfrak{a})=n_R(\mathfrak{a})=\max\{ \gamma_{R_\m}(I_\m) \,|\, I\subseteq R\textup{ an ideal},\, \m\in\maxspec R \textup{ s.t. }\mathfrak{a}_\m \neq R_\m \}.\] 
If \(R\) is local with maximal ideal \(\fm\) we write \(\kappa(R)\) for \(R/\fm\).
\end{definition}

\begin{theorem}\label{thm:finite_local_generators}
If \(R\) is a Noetherian commutative ring of Krull dimension at most \(1\) such that \(\textup{Q}(R)\) is Artinian and \(\mathfrak{a}\) is a fractional ideal of \(R\), then \(n(\mathfrak{a})\) is finite.
\end{theorem}
\begin{proof}
By Lemma~\ref{lem:artinian_quotient} we may assume \(R\) is local, so the theorem reduces to Theorem~3.1.2 of \cite{generators}.
\end{proof}

\begin{lemma}\label{lem:iota_ineq}
Suppose \(R\subseteq S\subseteq\textup{Q}(R)\) are commutative (sub)rings such that \(S\) is Noetherian as and \(R\)-module.
If \(\mathfrak{a}\) is a fractional ideal of \(R\), then \(n_S(S\mathfrak{a})\leq n_R(\mathfrak{a})\). 
\end{lemma}
\begin{proof}
Let \(\m\in\maxspec S\) be such that \((S\mathfrak{a})_\m\neq S_\m\). Then \(\p=R\cap \m\in\maxspec R\) with \(\mathfrak{a}_\p\neq R_\p\) by Lemma~\ref{lem:semi-local}.
Let \(I\subseteq S\). 
By Lemma~\ref{lem:quotient_length} there exists some regular \(r\in R\) such that \(rI\subseteq R\).
Since \(S_\m\) is an \(R_\p\)-algebra, we have \(\gamma_{S_\m}(I_\m)\leq \gamma_{R_\p}(I_\p) = \gamma_{R_\p}(rI_\p) \leq n_R(\mathfrak{a})\), from which the lemma follows.
\end{proof}

\begin{lemma}\label{lem:iota_bound_basis}
If \(Z\subseteq R\) are (sub)rings such that \(Z\) is Dedekind and \(R\) is projective as \(Z\)-module of local rank everywhere bounded by some constant \(m\), then \(n(\mathfrak{a})\leq m\) for every fractional ideal \(\mathfrak{a}\) of \(R\). 
\end{lemma}
\begin{proof}
It suffices to show that for each prime \(\p\subset Z\) and ideal \(I\subseteq R\) we have \(\gamma_{Z_\p}(I_\p)\leq m\).
Thus we may assume \(Z\) is local, and hence a principal ideal domain.  
Then \(R\) is free of rank at most \(m\).
Moreover, each \(Z\)-submodule of \(R\), in particular each ideal of \(R\), is free of rank at most \(m\).
\end{proof}

The following five lemmata are taken from \cite[Section~4]{lecture}.

\begin{lemma}\label{lem:min_gen}
Let \(R\) be a local commutative ring with maximal ideal \(\fm\) and \(A\) a local commutative \(R\)-algebra with maximal ideal \(A\fm\).
Let \(L\) and \(M\) be finitely generated \(R\)-modules and \(N\) a finitely generated \(A\)-module. 
Then:
\begin{enumerate}[label=\textup{(\roman*)}]
\item \(\gamma_R(L\tensor_R M)=\gamma_R(L)\cdot\gamma_R(M)\);
\item \(\gamma_R(N)=\gamma_R(A)\cdot\gamma_A(N)\);
\item \(\gamma_A(A\tensor_R M)=\gamma_R(M)\).
\end{enumerate}
\end{lemma}

\begin{lemma}\label{lem:cardinality_unit}
Let \(R\) be a local commutative ring and \(S\) a semi-local commutative \(R\)-algebra. 
Let \(M\subseteq S\) be an \(R\)-submodule such that \(M\not\subseteq \fm\) for all \(\fm\in\maxspec S\). If \(\#\kappa(R)\leq\#\maxspec S\), then \(M\cap S^*\) is non-empty.
\end{lemma}

\begin{lemma}\label{lem:large_extension}
Let \(R\) be local commutative ring with maximal ideal \(\fm\). 
Then for every \(n\in\Z_{\geq 0}\) there exists a commutative \(R\)-algebra \(A\) which is free of finite rank over \(R\), local with maximal ideal \(A\fm\), and such that \(\#\kappa(A)>n\cdot\rk_R(A)\).
\end{lemma}

\begin{lemma}\label{lem:module_has_one}
Let \(R\subseteq S\) be commutative (sub)rings with \(R\) local. 
Let \(M\subseteq S\) be an \(R\)-submodule such that \(1\in M\), and \((M^n:M^n)_S=R\) for some integer \(n+1\geq \gamma_R(\bigcup_{k\geq 0} M^k)\). 
Then \(M=R\).
\end{lemma}

\begin{lemma}\label{lem:tensor_division}
Let \(R\subseteq S\) be commutative (sub)rings, let \(A\) be a commutative \(R\)-algebra that is free as an \(R\)-module, and let \(M\subseteq S\) be an \(R\)-submodule. 
Then \((M\tensor_R A : M\tensor_R A)_{S\tensor A} = (M:M)_S \tensor_R A\). 
\end{lemma}

\begin{proposition}\label{prop:local_case}
Let \(R\) be a local Noetherian commutative ring of Krull dimension at most \(1\) for which \(\textup{Q}(R)\) is Artinian and let \(\mathfrak{a}\) be a fractional ideal of \(R\). If \(\mathfrak{a}^{n}:\mathfrak{a}^{n}=R\) for some \(n\in\Z_{>0}\) with \(n+1\geq n(\mathfrak{a})\), then \(\mathfrak{a}\) is an invertible ideal.
\end{proposition}
\begin{proof}
Let \(R\subseteq S\subseteq\textup{Q}(R)\) be such that \(S\) is Noetherian as \(R\)-module and \(S\mathfrak{a}\) is an invertible ideal of \(S\), which exists by Theorem~\ref{thm:exist-invertible-extension}. 
Note that \(S\) is semi-local by Lemma~\ref{lem:semi-local}.
There exists by Corollary~\ref{cor:invertible_descent} some invertible ideal \(\mathfrak{b}\) of \(R\) such that \(S\mathfrak{a} = S\mathfrak{b}\).
Note that \((\mathfrak{a}\mathfrak{b}^{-1})^n:(\mathfrak{a}\mathfrak{b}^{-1})^n=\mathfrak{a}^n:\mathfrak{a}^n=R\).
Then we may replace \(\mathfrak{a}\) by \(\mathfrak{a}\mathfrak{b}^{-1}\) to assume without loss of generality that \(S\mathfrak{a} = S\).

Consider now the \(R\)-algebra \(A\) from Lemma~\ref{lem:large_extension} satisfying  \(\#\maxspec S \cdot \rk_R(A) <\#\kappa(A)\), and write \(-_A\) for \(-\tensor_R A\). 
Since \(A\) is free over \(R\), we have \(\mathfrak{a}^{n}_A:\mathfrak{a}^{n}_A=R_A=A\) by Lemma~\ref{lem:tensor_division}. 
As \(S\) is Noetherian, \(S_A\) is free and Noetherian over \(S\).
If \(\mathfrak{a}_A\subseteq \mathfrak{m}\) for some \(\mathfrak{m}\in\maxspec S_A\), then by Lemma~\ref{lem:semi-local} also \(\mathfrak{a}=\mathfrak{a}_A\cap S\) is contained in the maximal ideal \(\mathfrak{m}\cap S\) of \(S\), which is impossible since \(S\mathfrak{a} = S\). 
Lastly, we have \(\#\maxspec S_A \leq \rk_R(A) \cdot \#\maxspec S < \#\kappa(A)\).

Now we may apply Lemma~\ref{lem:cardinality_unit} with \(R_A\), \(S_A\) and \(\mathfrak{a}_A\) in the roles of \(R\), \(S\) and \(M\) respectively, so we may choose some \(u\in \mathfrak{a}_A\cap (S_A)^*\).
Now let \(M=u^{-1} \mathfrak{a}_A\) and \(M^\infty=\bigcup_{k\geq 1} M^k\), and note that \(M^\infty\) is a Noetherian \(R\)-submodule of \(S_A\cong_R S^r\).
The projections \(S_A=S^r\to S^{r-1}\to\dotsm\to S^0\) induce projections \(M^\infty=N_r\to N_{r-1}\to\dotsm\to N_0\). 
Each kernel is isomorphic to a finitely generated \(R\)-submodule of \(\textup{Q}(R)\) and thus can be generated by \(n(\mathfrak{a})\) elements. 
Hence \(\gamma_R(M^\infty)\leq r n(\mathfrak{a})\) and \(\gamma_A(M^\infty)\leq n(\mathfrak{a})\leq n+1\) by Lemma~\ref{lem:min_gen}(ii).
Then by Lemma~\ref{lem:module_has_one} we have \(M=A\).
By Lemma~\ref{lem:min_gen}(iii) we have \(1=\gamma_A(M)=\gamma_A(\mathfrak{a}_A)=\gamma_R(\mathfrak{a})\), so \(\mathfrak{a}=Rx\) for some \(x\in\textup{Q}(R)\). 
As \(\textup{Q}(R)\cdot \mathfrak{a}=\textup{Q}(R)\) we must have that \(x\in\textup{Q}(R)^*\) and thus \(\mathfrak{a}\) is invertible.
\end{proof}

\begin{theorem}\label{thm:main_ideal}
Let \(R\) be a Noetherian commutative ring of Krull dimension at most \(1\) such that \(\textup{Q}(R)\) is Artinian and let \(\mathfrak{a}\) be a fractional ideal of \(R\).
Among all subrings \(R\subseteq S \subseteq\textup{Q}(R)\) such that \(S\mathfrak{a}\) is an invertible ideal of \(S\), there exists a unique minimum with respect to inclusion.
This ring is Noetherian as \(R\)-module and equals \(\mathfrak{a}^{n}:\mathfrak{a}^{n}\) for all \(n+1\geq n(\mathfrak{a})\).
\end{theorem}
\begin{proof}
Let \(S=\mathfrak{a}^n:\mathfrak{a}^n\) for some $n+1\geq n(\mathfrak{a})$ and let \(\mathfrak{b}=S\mathfrak{a}\). 
We will show that \(\mathfrak{b}\) is invertible.
Note that \(S\) is Noetherian as \(R\)-module, and that we have \(\mathfrak{b}^n=S\mathfrak{a}^n=\mathfrak{a}^n\) and \(\mathfrak{b}^n:\mathfrak{b}^n=\mathfrak{a}^n:\mathfrak{a}^n=S\).
Let \(\fm\in\maxspec S\).
From Lemma~\ref{lem:arith_loc} we deduce that \(\mathfrak{b}_\fm^n:\mathfrak{b}_\fm^n=S_\fm\).
We have \(n(\mathfrak{b})\leq n(\mathfrak{a})\leq n+1\) by Lemma~\ref{lem:iota_ineq}, so \(\mathfrak{b}_\fm\) is invertible by Proposition~\ref{prop:local_case}.
Hence \(\mathfrak{b}\) is invertible by Lemma~\ref{lem:inv_loc_prop}.

Suppose \(R\subseteq T \subseteq \textup{Q}(R)\) is such that \(T\mathfrak{a}\) is an invertible ideal of \(T\). Then \(T\cdot (\mathfrak{a}^n:\mathfrak{a}^n) \subseteq (T\mathfrak{a})^n:(T\mathfrak{a})^n = T\) by Lemma~\ref{lem:inv_basic}, so \(\mathfrak{a}^n:\mathfrak{a}^n\subseteq T\). Hence \(\mathfrak{a}^n:\mathfrak{a}^n\) is the unique minimum.  
\end{proof}

Using the theory of this section, we can construct two functorial polynomial-time algorithms. 
Recall that an order is a commutative ring \(R\) which is free of finite rank as \(\Z\)-module, and in particular that $R$ is allowed to contain zero-divisors.
Note that for an order $R$ we have $\Q R=Q(R)$. 

\begin{theorem}\label{thm:blowup}
There exists a polynomial-time algorithm that, given an order \(R\) and a fractional \(R\)-ideal \(\mathfrak{a}\), computes the unique minimal order \(R\subseteq S\subseteq \Q R\) such that \(S\mathfrak{a}\) is invertible.
\end{theorem}
\begin{proof}
We iteratively compute \(\mathfrak{a}^k\) for \(1\leq k\leq n = \rk_\Z(R)-1\) and return \(\mathfrak{a}^n:\mathfrak{a}^n\). 
By Lemma~\ref{lem:iota_bound_basis} and Theorem~\ref{thm:main_ideal} the output is correct.
\end{proof}

\begin{example}
Although they would be unusual rings to consider in the context of algorithms, Theorem~\ref{thm:main_ideal} and hence an analogue to Theorem~\ref{thm:blowup} apply to some rings that are not free as $\Z$-module, such as $\Z\times(\Z/5\Z)$ or $\Z[\varepsilon]/(5,\varepsilon^2)$. Note that there are also commutative rings which are finitely generated as $\Z$-module to which Theorem~\ref{thm:main_ideal} does not apply, such as $R=\Z[\varepsilon]/(5\varepsilon,\varepsilon^2)$, because $Q(R)$ needs to be Artinian. 
\end{example}

A result we will now prove and that to the author's knowledge is new is an algorithm to detect and remove torsion in $\mathcal{I}(R)$.

\begin{lemma}\label{lem:torsion_equiv}
Let $R$ be a commutative ring and let $\mathfrak{a}\in\mathcal{I}(R)$.
If $\mathfrak{a}$ is torsion, then there exists some subring $R\subseteq S\subseteq\textup{Q}(R)$ such that $S$ is finitely generated as $R$-module and such that $S\mathfrak{a}=S$.
If $R$ is an order, then the converse holds.
\end{lemma}
\begin{proof}
Suppose $\mathfrak{a}^\ell=R$ for $\ell>0$ and let $S$ be the ring generated by $R$ and $\mathfrak{a}$.
Then $S\supseteq S\mathfrak{a} \supseteq \dotsm \supseteq S\mathfrak{a}^\ell = S$ and hence $S=S\mathfrak{a}$ hold. 

Suppose now that $R$ is an order and that $S\subseteq\textup{Q}(R)$ satisfies and $S\mathfrak{a}=S$. 
Write $\mathfrak{f}=R:S$, which is a fractional ideal of $S$ contained in $R$, and note that for each $k\geq0$ we have $\mathfrak{f}=\mathfrak{f}\mathfrak{a}^k\subseteq\mathfrak{a}^k\subseteq S$.
If $S$ is finitely generated over $R$, then $S/\mathfrak{f}$ is finite, so $\mathfrak{a}^k$ can take only finitely many values as $k$ varies. Invertibility of $\mathfrak{a}$ then implies $\mathfrak{a}$ must be torsion.
\end{proof}

\begin{proposition}\label{prop:torsion}
There exists a polynomial-time algorithm that, given an order \(R\) and an invertible ideal \(\mathfrak{a}\) of \(R\), decides whether \(\mathfrak{a}\) is torsion, and if so computes the minimal order \(R\subseteq S\subseteq\Q R\) such that \(S\mathfrak{a}=S\).
\end{proposition}
\begin{proof}
We compute using Theorem~\ref{thm:blowup} the minimal order \(R\subseteq S\subseteq\Q R\) where \(R+\mathfrak{a}\) becomes invertible. We claim \(\mathfrak{a}\) is torsion if and only if \(S\mathfrak{a}=S\).

(\(\Rightarrow\)) 
As \(S\mathfrak{a}\) is the quotient of \(\mathfrak{b}=(S+(S\mathfrak{a})^{-1})^{-1}\) and \(\mathfrak{c}=(S+S\mathfrak{a})^{-1}\) as in Lemma~\ref{lem:ideal_as_fraction} with \(\mathfrak{b}+\mathfrak{c}=S\), the ideal \(S\mathfrak{a}\) is torsion if and only if \(\mathfrak{b}\) and \(\mathfrak{c}\) are.
However, since \(\mathfrak{b},\mathfrak{c}\subseteq S\), this is only possible if \(\mathfrak{b}=\mathfrak{c}=S\). 
Hence \(S\mathfrak{a}=S : S=S\).

(\(\Leftarrow\))
This follows from Lemma~\ref{lem:torsion_equiv}.

Finally, suppose that \(\mathfrak{a}\) is torsion and \(R\subseteq T\subseteq\Q R\) is an order such that \(T\mathfrak{a}=T\). Then \(T(R+\mathfrak{a})=T\) is invertible over \(T\), so \(S\subseteq T\). Hence \(S\) is the unique minimal order where \(S\mathfrak{a}=S\).
\end{proof}

\section{Factor refinement}\label{sec:factor_refinement}

In this section we similarly generalize the results of Ge \cite{Ge} and Buchmann--Eisenbrand \cite{Buchmann-Eisenbrand} on factorizations of fractional ideals to include the case where the ring is allowed to contain zero divisors.
For this we use the generalized blowup from Section~\ref{sec:blowing_up}.
One can prove the main result in this section by following \cite{Buchmann-Eisenbrand}, mutatis mutandis.
To prevent this section from being a mechanical verification of this proposed approach, we offer instead an alternative proof, following \cite{thesis}.

We consider a commutative ring $R$ and invertible ideals contained within it.
Any two such ideals are called \emph{coprime} if their sum equals $R$, and we leave coprimality for non-invertible ideals, and for invertible ideals not contained in $R$, undefined.
It is easy to see that for a set \(C\) of pairwise coprime invertible ideals strictly contained in \(R\) the natural map \(\Z^{(C)}\to\mathcal{I}(R)\) is injective with image \(\langle C\rangle\) and that \(\langle C\rangle\) is closed under addition. In fact, \(\Z^{(C)}\to\langle C\rangle\) is an isomorphism of partially ordered groups. 

\begin{lemma}\label{lem:trial_division}
There exists a polynomial-time algorithm that, given an order \(R\) and a finite set \(C\) of pairwise coprime invertible ideals strictly contained in \(R\) and some invertible ideal \(\mathfrak{a}\) of \(R\), decides whether \(\mathfrak{a}\) is in the image of the injection \(\Z^{(C)}\to \mathcal{I}(R)\) and if so computes the preimage.
\end{lemma}
\begin{proof}
First verify whether \(R+\mathfrak{a}\) is invertible using Theorem~\ref{thm:blowup}, as it should be if \(\mathfrak{a}\in \langle C\rangle\). 
If so, then by Lemma~\ref{lem:ideal_as_fraction} we may write \(\mathfrak{a}=\mathfrak{b}:\mathfrak{c}\) for invertible \(\mathfrak{b},\mathfrak{c}\subseteq R\). We then proceed using trial division.
\end{proof}

\begin{definition}
Let \(R\) be a commutative ring. The set of fractional ideals of \(R\) has the structure of a multiplicative monoid with unit \(R\).
For a set \(X\) of fractional ideals contained in \(R\), we write \(\lla X \rra\) for the submonoid generated by \(X\), and we define the {\em closure of \(X\)}, written \(\textup{cl}_R(X)\) or simply \(\textup{cl}(X)\), to be the smallest submonoid containing \(X\) such that for all \(\mathfrak{a},\mathfrak{b}\in \textup{cl}(X)\) we have \(\mathfrak{a}+\mathfrak{b}\in\textup{cl}(X)\), and if \(\mathfrak{b}\) is invertible and \(\mathfrak{a}:\mathfrak{b}\subseteq R\) also \(\mathfrak{a}:\mathfrak{b}\in\textup{cl}(X)\).
\end{definition}

\begin{lemma}
Let \(R\subseteq S \subseteq\textup{Q}(R)\) be commutative (sub)rings and let \(X\) be a set of fractional ideals contained in \(R\).
Writing \(S\cdot X=\{S\mathfrak{a}\,|\, \mathfrak{a}\in X\}\), we have \(S\cdot \textup{cl}_R(X)\subseteq \textup{cl}_S(S\cdot X)\). \qed
\end{lemma}
\begin{proof}
Consider $M=\{\mathfrak{a}\subseteq R\,|\, S\cdot \mathfrak{a} \in \textup{cl}_S(S\cdot X)\}$, and note that it is a multiplicative monoid closed under addition, and under division by invertible elements for which the quotient is integral. Hence $\textup{cl}_R(X)\subseteq M$ by definition of the closure, and $S\cdot \textup{cl}_R(X)\subseteq S\cdot M \subseteq \textup{cl}_S(S\cdot X)$.
\end{proof}

\begin{lemma}\label{lem:closure_is_monoid}
Let \(R\) be a commutative ring and \(C\) a set of invertible ideals contained in \(R\) which are pairwise coprime. Then \(\textup{cl}(C)=\lla C\rra\).
\end{lemma}
\begin{proof}
Clearly \(\lla C\rra \subseteq \textup{cl}(C)\).
For all \(\mathfrak{a},\mathfrak{b}\in\lla C\rra\) we may write \(\mathfrak{a}=\prod_{\mathfrak{c}\in C} \mathfrak{c}^{a_\mathfrak{c}}\) and \(\mathfrak{b}=\prod_{\mathfrak{c}\in C}\mathfrak{c}^{b_\mathfrak{c}}\) with \(a_\mathfrak{c},b_\mathfrak{c}\in\Z_{\geq0}\) almost all \(0\). 
Then
\[
\mathfrak{a}+\mathfrak{b} = \prod_{\mathfrak{c}\in C} \mathfrak{c}^{\min\{a_\mathfrak{c},b_\mathfrak{c}\}}
\quad\text{and}\quad 
\mathfrak{a}:\mathfrak{b} = \prod_{\mathfrak{c}\in C} \mathfrak{c}^{a_\mathfrak{c}-b_\mathfrak{c}}. \]
Hence \(\mathfrak{a}+\mathfrak{b}\in \lla C\rra\), and \(\mathfrak{a}:\mathfrak{b}\in \lla C\rra\) if \(\mathfrak{a}:\mathfrak{b}\subseteq R\). 
Thus \(\lla C\rra = \textup{cl}(C)\).
\end{proof}

\begin{definition}
Let \(R\) be a commutative ring and \(X\) a set of fractional ideals contained in \(R\).
A {\em coprime basis}\index{coprime!-basis} for \(X\) is a set \(C\) of pairwise coprime invertible ideals strictly contained in \(R\) satisfying \( X\subseteq\lla C\rra\).
\end{definition}

Note that a coprime basis need not exist for every \(X\). At the very least, the ideals in \(X\) should be invertible.

\begin{lemma}\label{lem:coprime_basis_partial_order}
For a commutative ring \(R\) and a set \(X\) of fractional ideals contained in \(R\) we may equip the set of coprime bases of \(X\) with a partial order where \(C\leq D\) if and only if \(\lla C\rra\subseteq\lla D\rra\).
\end{lemma}
\begin{proof}
It suffices to verify for coprime bases \(C\) and \(D\) that \(\lla C\rra = \lla D\rra\) implies \(C=D\).
For each $\mathfrak{c}\in C$ write \(\mathfrak{c}=\prod_{\mathfrak{d}\in D} \mathfrak{d}^{m_{\mathfrak{cd}}}\) for some \(m_{\mathfrak{c}\mathfrak{d}}\in\Z_{\geq 0}\).
Since the elements of \(C\) are pairwise coprime, there is for every \(\mathfrak{d}\in D\) at most one \(\mathfrak{c}\in C\) such that \(m_{\mathfrak{cd}}>0\).
Because \(\lla D\rra\subseteq \lla C\rra\), there is no \(\mathfrak{d}\in D\) such that for all \(\mathfrak{c}\in C\) we have \(m_{\mathfrak{cd}}=0\). 

Let \(\mathfrak{d}\in D\). Then there exist \(\mathfrak{c}\in C\) and \(m>0\) such that \(\mathfrak{c}=\mathfrak{d}^m\) and in turn by symmetry \(\mathfrak{e}\in D\) and \(n>0\) such that \(\mathfrak{e}=\mathfrak{c}^n\). Then \(\mathfrak{e}=\mathfrak{d}^{mn}\), so \(\mathfrak{e}=\mathfrak{d}\) and \(m=n=1\). Thus \(\mathfrak{d}=\mathfrak{c}\in C\) and \(D\subseteq C\). By symmetry we have \(C=D\).
\end{proof}

\begin{proposition}\label{prop:closure}
Let \(R\) be a Noetherian commutative ring and \(X\) a set of fractional ideals contained in \(R\). Then:
\begin{enumerate}[nosep,label=\textup{(\roman*)}]
\item \(X\) has a coprime basis if and only if \(\textup{cl}(X)\subseteq\mathcal{I}(R)\);
\item if \(X\) has a coprime basis, then it has a unique minimal one;
\item if \(C\) is a coprime basis of \(X\), then \(C\) is minimal if and only if \(C\subseteq\textup{cl}(X)\).
\end{enumerate}
\end{proposition}
\begin{proof}
(i) Suppose \(X\) has a coprime basis \(D\).
Then \(X\subseteq \lla D\rra = \textup{cl}(D)\) by Lemma~\ref{lem:closure_is_monoid}, so \(\textup{cl}(X)\subseteq\textup{cl}(D)=\lla D\rra \subseteq\mathcal{I}(R)\). 
Suppose instead that \(\textup{cl}(X)\subseteq \mathcal{I}(R)\).
We will show that
\[C=\{ \mathfrak{a}\in \textup{cl}(X) \,|\, \forall\, \mathfrak{b}\in \textup{cl}(X),\, \mathfrak{a}\subsetneq \mathfrak{b} \Leftrightarrow \mathfrak{b}=R\}\]
is a coprime basis of \(X\).

First, note that the elements of \(C\) are pairwise coprime: For \(\mathfrak{a},\mathfrak{b}\in C\) we have \(\mathfrak{a}+\mathfrak{b}\in \textup{cl}(X)\). If \(\mathfrak{a}\subsetneq \mathfrak{a}+\mathfrak{b}\), then \(\mathfrak{a}+\mathfrak{b}=R\) by definition of \(C\), and similarly when \(\mathfrak{b}\subsetneq\mathfrak{a}+\mathfrak{b}\). Otherwise \(\mathfrak{a}=\mathfrak{a}+\mathfrak{b}=\mathfrak{b}\).
Second, we show \(\textup{cl}(X)\subseteq\lla C\rra\) using Noetherian induction:
Certainly \(R\in\lla C\rra\). 
Now let \(\mathfrak{a}\in \textup{cl}(X)\setminus\{R\}\) and suppose \(\mathfrak{c}\in\lla C\rra\) for all \(\mathfrak{c}\in \textup{cl}(X)\) with \(\mathfrak{a}\subsetneq \mathfrak{c}\).
Either \(\mathfrak{a}\in C\), or there is some \(\mathfrak{b}\in \textup{cl}(X)\) such that \(\mathfrak{a}\subsetneq \mathfrak{b} \subsetneq R\), in which case \(\mathfrak{b}, (\mathfrak{a}:\mathfrak{b})\in\lla C\rra\) by the induction hypothesis and hence \(\mathfrak{a}\in\lla C\rra\). Thus \(C\) is a coprime basis for \(X\), as was to be shown.

(ii) Suppose now that \(X\) has a coprime basis. 
We will show that \(C\) as in (i) is the unique minimal coprime basis.
Let \(D\) be any coprime basis of \(X\).
We have \(C\subseteq\textup{cl}(X)\), so \(\textup{cl}(C) \subseteq \textup{cl}(X)\).
On the other hand, \(X\subseteq \textup{cl}(C)\), so \(\textup{cl}(C)=\textup{cl}(X)\).
Similarly for \(D\) we have \(\textup{cl}(X)\subseteq\textup{cl}(D)\).
Hence \(\lla C \rra =\textup{cl}(C)=\textup{cl}(X)\subseteq\textup{cl}(D)=\lla D\rra\) by Lemma~\ref{lem:closure_is_monoid}.
Thus \(C\leq D\), as was to be shown.

(iii) It is clear that the minimal coprime basis from (ii) satisfies \(C\subseteq\textup{cl}(X)\).
Let \(D\) be any coprime basis of \(X\) such that \(D\subseteq\textup{cl}(X)\).
Then as before we obtain \(\lla D\rra = \textup{cl}(D)=\textup{cl}(X)\).
Hence \(\lla C \rra = \textup{cl}(X)=\lla D\rra\) and \(C=D\) by Lemma~\ref{lem:coprime_basis_partial_order}. Hence \(D\) is minimal.
\end{proof}

We now construct an algorithm to compute the minimal coprime base when \(R\) is an order. 
The output of such an algorithm is clearly functorial. 

\begin{theorem}\label{thm:coprime_basis}
There exists a polynomial-time algorithm that, given an order \(R\) and a finite set \(X\) of fractional ideals contained in \(R\), computes the minimal order \(R\subseteq S\subseteq \Q R\) such that \(\{S\mathfrak{a} \,|\, \mathfrak{a}\in X \}\) has a coprime basis, and then computes the minimal coprime basis of \(\{S\mathfrak{a} \,|\, \mathfrak{a}\in X \}\).
\end{theorem}
\begin{proof}
Start with \(S\) equal to the minimal order \(R\subseteq S\subseteq\textup{Q}(R)\) where the elements of \(X\) become invertible using Theorem~\ref{thm:blowup}, and let \(C=X\).

Iteratively compute \(\mathfrak{c}=S\mathfrak{a}+S\mathfrak{b}\) for distinct \(\mathfrak{a},\mathfrak{b}\in C\). 
If \(\mathfrak{c}\neq S\), replace \(S\) by the unique minimal order \(S\subseteq T\subseteq \textup{Q}(R)\)  where \(T\mathfrak{c}\) is invertible using Theorem~\ref{thm:blowup}, replace \(\mathfrak{a}\) and \(\mathfrak{b}\) in \(C\) by \(T\mathfrak{a}: T\mathfrak{c}\) and \(T\mathfrak{b}:T\mathfrak{c}\), and add \(T\mathfrak{c}\) to \(C\).
Once \(S\mathfrak{a}+S\mathfrak{b}=S\) for all distinct \(\mathfrak{a},\mathfrak{b}\in C\) we terminate and return the order \(S\) and coprime basis \(S\cdot C\).

Polynomial run time follows from Lemma~\ref{lem:local_length} and the fact that \(\sum_{\mathfrak{a}\in C} \ell_R(S/\mathfrak{a}S)\), which is bounded by the length of the input, decreases by at least \(1\) after every iteration where \(\mathfrak{c}\neq S\). 
For this, the fact that \(S\) changes throughout the algorithm is irrelevant by Proposition~\ref{prop:invertible_ideal_finite_extension}. 
This proposition also gives a polynomial bound on \(\# C\) and hence on the number of pairs \(\mathfrak{a},\mathfrak{b}\in C\) to check for coprimality every iteration.

It remains to show correctness. 
With induction one shows that during each step of the algorithm \(S\cdot X\subseteq\lla S\cdot C\rra\) holds.
It follows that \(S\cdot C\) is indeed a coprime basis for \(S\cdot X\), and \(S\cdot C\subseteq\textup{cl}(S\cdot X)\), so it is minimal by Proposition~\ref{prop:closure}.
Suppose \(R\subseteq T\subseteq\textup{Q}(R)\) be such that \(\textup{cl}(TX)\subseteq\mathcal{I}(T)\).
Then at every step of the algorithm we could replace \(S\) by \(S\cap T\) and preserve invertibility, so \(S\subseteq T\) holds at every step by the minimality of \(S\) guaranteed by Theorem~\ref{thm:blowup}.
Hence \(S\) is minimal such that \(\textup{cl}(S\cdot X)\subseteq\mathcal{I}(S)\), and the algorithm is correct.
\end{proof}

\section{Roots of ideals}

In this section we will prove the main theorems on taking roots of ideals in orders.
For the remainder of the section \(R\) will be a commutative ring.

Let \(S\) be a commutative \(R\)-algebra. 
We say $S$ is \emph{separable} if $S$ is projective as $(S\tensor_RS)$-module, or equivalently if the multiplication map $S\tensor_R S\to S$ treated as $(S\tensor_R S)$-module homomorphism has a section.
We say \(S\) is {\em finite-\'etale}\index{finite-\'etale} if \(S\) is projective and separable over \(R\).

\begin{lemma}
\label{lem:etale_base_change}
Let \(S\) be a finite-\'etale \(R\)-algebra. Then
\begin{enumerate}[nosep,label=\textup{(\roman*)}]
\item for all ideals \(\mathfrak{a}\subseteq R\) the \(R/\mathfrak{a}\)-algebra \(S/\mathfrak{a}S\) is finite-\'etale;
\item for all \(\mathfrak{m}\in\maxspec R\) the \(R_\mathfrak{m}\)-algebra \(S_\mathfrak{m}\) is finite-\'etale;
\item if \(R\) is a field, then \(S\) is a product of fields.
\end{enumerate}
\end{lemma}
\begin{proof}
For (i) and (ii) it suffices to verify separability.
For (i) it is trivial that \(R/\mathfrak{a}\) is separable over \(R\), hence \(S/\mathfrak{a}S\) is separable over \(R/\mathfrak{a}\) by Proposition~III.1.7 of \cite{separable}.
For (ii) we have Proposition~III.2.5 of \cite{separable}.
Finally, (iii) is a consequence of Theorem~III.3.1 of \cite{separable}.
\end{proof}

\begin{theorem}\label{thm:root_exists_condition}
Let \(Z\subseteq R\subseteq S \subseteq \textup{Q}(R)\) be commutative (sub)rings such that \(Z\) is Dedekind and \(S\) is finitely generated as a \(Z\)-module, let \(\mathfrak{a}\subseteq R\) be an invertible ideal and let \(m\in\Z_{\geq 0}\).
Write \(a=\mathfrak{a}\cap Z\) and suppose \(R/\mathfrak{a}\) is finite-\'etale over \(Z/a\).
If there exists an ideal \(\mathfrak{b}\subseteq S\) with \(S\mathfrak{a}=\mathfrak{b}^m\), then there exists an ideal \(b\subseteq Z\) with \(a=b^m\).
\end{theorem}

\begin{proof}
It suffices to prove the theorem for local \(Z\): all conditions on the rings and ideals are preserved by localization at a prime of \(Z\), which for the finite-\'etale property is Lemma~\ref{lem:etale_base_change}.(ii), and the conclusion holds if it holds everywhere locally.
If \(Z\) is a field, then \(Z=R=S=\textup{Q}(R)\) and the theorem holds trivially.
Thus we may assume \(Z\) is a discrete valuation ring with maximal ideal \(p=\pi Z\). 
Note that \(Z\) is Noetherian, hence \(S\) and consequently \(R\) are Noetherian \(Z\)-modules and in particular Noetherian rings.

Suppose \(\mathfrak{b}\subseteq S\) is such that \(S\mathfrak{a}=\mathfrak{b}^m\).
Write \(a=p^k\) for some \(k\geq 0\). To show there exists an ideal \(b\) with \(a=b^m\), it suffices to show that \(m \mid k\).
We may assume that \(k>0\), otherwise this is trivial.
By Lemma~\ref{lem:semi-local} and Proposition~\ref{prop:semi_loc_inv} we have \(\mathfrak{a}=\alpha R\) and \(\mathfrak{b}=\beta S\) for some regular \(\alpha\in R\) and \(\beta\in S\) .

By Lemma~\ref{lem:quotient_length} we have \(\ell_R(S/R)<\infty\), so by Proposition~\ref{prop:invertible_ideal_finite_extension} we have \([R/\alpha R]_R=[S/\alpha S]_R\).
We have inclusions \(S\supseteq \beta S \supseteq \dotsm \supseteq \beta^m S=\alpha S\).
For all \(i\) we have an isomorphism \(S/\beta S\to \beta^i S/\beta^{i+1}S\) since \(\beta^i\) is regular, so \([R/\alpha R]_R=[S/\alpha S]_R=m \cdot [S/\beta S]_R\).

Write \(A_i=\pi^i\cdot (R/\alpha R)\).
We have inclusions \(A_0 \supseteq A_1 \supseteq\dotsm\supseteq A_k=0\).
Because \(R/\alpha R\) as \(Z/\pi^k Z\)-algebra is finite-\'etale by assumption, it is projective and hence free.
Therefore multiplication by \(\pi^i\) for \(0\leq i < k\) is an isomorphism \(A_0/A_1\to A_{i}/A_{i+1}\) of \(Z\)-modules and hence of \(R\)-modules.
We conclude that \(k\cdot [A_0/A_1]_R = [R/\alpha R]_R = m \cdot [S/\beta S]_R\).
By Lemma~\ref{lem:etale_base_change} the \(Z/\pi Z\)-algebra \(A_0/A_1\) is finite-\'etale and thus a product of fields. 
In particular, if we choose any \(\m\in\maxspec R\) containing \(\alpha R+\pi R\) we obtain \([A_0/A_1]_R(\m)=\ell_{R_\m}((A_0/A_1)_\m)=1\).
It follows that \(k=m\cdot [S/\beta S]_R(\m)\), as was to be shown.
\end{proof}

\begin{example}
Under the assumptions of Theorem~\ref{thm:root_exists_condition} it need not be the case that \(\mathfrak{a}\) itself be an \(m\)-th power in \(\mathcal{I}(R)\).

Let \(R=\Z[2\sqrt{2}]\) and \(\mathfrak{a}=(2+2\sqrt{2})R\). 
Then \(R/\mathfrak{a}\cong \Z/a\) as \(\Z/a\)-algebra for \(a=\mathfrak{a}\cap \Z=4\Z\), so \(R/\mathfrak{a}\) is certainly finite-\'etale over \(\Z/a\).
Since \(1+\sqrt{2}\) is a unit in the maximal order \(S=\Z[\sqrt{2}]\), we have that \(S\mathfrak{a}=(S\sqrt{2})^2\).
Suppose \(\mathfrak{c}\in\mathcal{I}(R)\) satisfies \(\mathfrak{c}^2=\mathfrak{a}\).
Square roots of ideals in \(S\) are unique, so \(S\mathfrak{c}=S\sqrt{2}\) and \(\mathfrak{c}\subseteq S\sqrt{2}\).
On the other hand we have \(\mathfrak{c}=\mathfrak{a}\cdot (R:\mathfrak{c}) \supseteq \mathfrak{a} \cdot (S2 : S\sqrt{2}) = 2\sqrt{2} S\).
Thus \(\mathfrak{c}\) corresponds to some \(R\)-submodule \(\mathfrak{d}\) of \(S/2S\) with square \((1+\sqrt{2})R+2S\).
Clearly \(\mathfrak{d}\neq S/2S\), so \(\mathfrak{d}=d R+2S\) for some \(d\in S/2S\).
As \(d^2\in\{0,1\}\) we conclude that \(\mathfrak{d}^2\neq (1+\sqrt{2})R+2S\), so no such \(\mathfrak{c}\) exists.
\end{example}

We now proceed on to Theorem~\ref{mainthm:max_ideal_root}. First we give a version of a theorem by Cioc\u{a}nea-Teodorescu \cite{Iuliana} fit for our purposes.

\begin{proposition}
\label{prop:alg_etale_or_else}
There exists a polynomial-time algorithm that, given a morphism $f:R\to S$ of finite commutative rings,
decides whether \(S\) is finite-\'etale over \(R\) and if not, computes either some ideal \(0\subsetneq \mathfrak{a}\subsetneq R\) or some ideal \(0\subsetneq \mathfrak{b}\subsetneq S\).

The algorithm is functorial in the sense that if we have an isomorphism $f\to f'$ between morphisms $f:R\to S$ and $f':R'\to S'$, i.e., a pair of isomorphisms $g:R\to R'$ and $h:S\to S'$ such that $f'\circ g = h\circ f$, then the outputs for $f$ and $f'$ are either both ideals $\mathfrak{a}$ and $\mathfrak{a}'$ from $R$ and $R'$ respectively such that $g(\mathfrak{a})=\mathfrak{a}'$, or similarly for $S$ and $S'$ and $h$.
\end{proposition}
\begin{proof}
Projectivity over finite rings can be tested using Theorem 5.4.1 from \cite{Iuliana}, hence the finite-\'etale property can be tested. Suppose now that \(S\) is not finite-\'etale over \(R\).
Then \(S\) is not free as \(R\)-module or, by Proposition 6.2.14.ii from \cite{Iuliana}, \(S\) is not separable over \(\Z\).

We may compute a set of generators of \(S\) of minimal cardinality \(n\) by Theorem~4.1.3 from \cite{Iuliana}. 
In turn, we may compute the \((n-1)\)-th Fitting ideal \(\mathfrak{a}\subsetneq R\) of \(S\) in polynomial time from Definition~20.4 in \cite{Eisenbud}.
If \(S\) is not free, we may return the ideal \(\mathfrak{a}\), since it is non-zero by Proposition~20.8 in \cite{Eisenbud} and functorial by Corollary~20.4 in \cite{Eisenbud}. 
If \(S\) is not separable over \(\Z\), we obtain a ideal \(0\subsetneq \mathfrak{b} \subsetneq S\) from Proposition 6.1.3 from \cite{Iuliana} which is functorial as well.
\end{proof}

Consider the category $\textup{Frc}$ whose objects are pairs $(R,\mathfrak{a})$, where $R$ is an order and $\mathfrak{a}$ is a fractional ideal of $R$, and where the morphisms $(R,\mathfrak{a})\to(R',\mathfrak{a}')$ are the ring isomorphisms $f:\Q R\to \Q R'$ such that $f(R)=R'$ and $f(\mathfrak{a})=\mathfrak{a}'$. 
For $n\in\Z_{\geq 0}$, an \emph{$n$-th root} of $(R,\mathfrak{a})\in\textup{Frc}$ is an $(S,\mathfrak{b})\in\textup{Frc}$ such that $R\subseteq S\subseteq \Q R$ and $\mathfrak{b}^n=S\mathfrak{a}$.

\begin{theorem}\label{thm:max_ideal_root}
There exists a functor $F:\textup{Frc}\to\Z_{\geq0}\times \textup{Frc}$ such that for the objects $A\in\textup{Frc}$ and $(n,B)=F(A)$ we have that $B$ is an $n$-th root of $A$ such that $n$ is maximal with respect to divisibility among all roots of $A$, and for the morphisms $f:A\to A'$ we have $F(f)=\id \times f$. 
Moreover, there exists a polynomial-time algorithm that takes as input objects $A,A'\in\textup{Frc}$ and a morphism $f:A\to A'$, and computes $F(A)$, $F(A')$ and $F(f)$.
\end{theorem}
Note that \(n=0\) corresponds to the case where \(\mathfrak{a}\) is torsion.
\begin{proof} 
We will describe the functor in terms of the polynomial-time algorithm that computes it.
To produce the algorithm it clearly suffices to produce an algorithm that computes the functor on a single object, so suppose $A=(R,\mathfrak{a})$ is given. 

First compute some order \(R\subseteq S\subseteq \textup{Q}(R)\) such that \(S\mathfrak{a}\) and \(S+S\mathfrak{a}\) are invertible using Theorem~\ref{thm:blowup}.
Then write \(S\mathfrak{a}=\mathfrak{a}_+:\mathfrak{a}_-\) with \(\mathfrak{a}_+,\mathfrak{a}_-\subseteq S\) invertible and coprime as in Lemma~\ref{lem:ideal_as_fraction}, and apply the algorithm recursively to \(\mathfrak{a}_+\) and \(\mathfrak{a}_-\) separately with \(S\) in the place of \(R\).
Since the ideals are coprime, we may obtain a solution \(n=\gcd(n_+,n_-)\) from solutions \(n_+\) and \(n_-\) for \(\mathfrak{a}_+\) and \(\mathfrak{a}_-\) respectively, and similarly we may construct \(B\).
Hence we may now assume that \(\mathfrak{a}\subsetneq R\).

Suppose that at some point during the algorithm we obtain an ideal \(\mathfrak{a}\subsetneq\mathfrak{d}\subsetneq R\).
Then compute some extension \(T\) and a coprime basis \(C\) for \(\{T\mathfrak{a},T\mathfrak{d}\}\) using Theorem~\ref{thm:coprime_basis}.
Using Lemma~\ref{lem:trial_division} we may write \(T\mathfrak{a}=\prod_{\mathfrak{c}\in C} (T\mathfrak{c})^{m_{\mathfrak{c}}}\) for some \(m_\mathfrak{c}\in\Z_{\geq 0}\).
As before we may solve the problem by applying the algorithm recursively to all \(\mathfrak{c}\in C\).
By the assumption on \(\mathfrak{d}\) we have \(\mathfrak{a}\subsetneq \mathfrak{c}\) for all \(\mathfrak{c}\in C\), so the recursion is well-founded.

Now we proceed to the actual algorithm.
Compute \(a\in\Z_{>1}\) such that \(\mathfrak{a}\cap\Z=a\Z\).
By Proposition~\ref{prop:alg_etale_or_else} we may assume that \(R/\mathfrak{a}\) is finite-\'etale over \(\Z/a\Z\), otherwise we can proceed recursively as above.
Then write \(a=b^m\) for some \(b,m\in\Z_{>0}\) with \(m\) maximal.
If \(b\not\in\mathfrak{a}\), then we may proceed recursively with \(\mathfrak{d}=bR+\mathfrak{a}\). 
Otherwise \(b\in a\Z\), so \(a=b\) and \(m=1\), in which case the solution is \(n=1\) and \(\mathfrak{b}=\mathfrak{a}\) by Theorem~\ref{thm:root_exists_condition}.

That the algorithm is functorial and runs in polynomial time follows from all theorems applied.
\end{proof}

\begin{corollary}
There exists a functorial polynomial-time algorithm that, given an order \(R\) in a number field, a fractional ideal \(\mathfrak{a}\) of \(R\) and a positive integer \(n\), decides whether there exist an order \(R\subseteq S\subseteq \Q R\) and fractional ideal \(\mathfrak{b}\) of \(S\) such that \(\mathfrak{b}^n=S\mathfrak{a}\) and if so computes such \(S\) and \(\mathfrak{b}\). \qed 
\end{corollary}

The algorithm of Theorem~\ref{thm:max_ideal_root} differs from that of Belabas--Simon \cite{Belabas} mainly in that we allow arbitrary orders, while \cite{Belabas} requires that the order be the maximal order of a number field.
Even in the case of a maximal order of a number field, the algorithm of \cite{Belabas} requires additional ramification data of which it is unclear if it can be computed in polynomial time given just the maximal order.
Thus our proof offers an improvement also in the specialized case.
The practical performance of the two algorithms could be quite similar on the inputs where both apply, as there the algorithms behave essentially the same, factoring the ideal $\mathfrak{a}$ depending on the structure of $R/\mathfrak{a}$ as $\Z/(\mathfrak{a}\cap\Z)$-algebra.

We finish by highlighting some ways in which root-taking in an arbitrary order is essentially different from root-taking in a maximal order. 

\begin{example}\label{ex:need_extension}
A fractional ideal of an order \(R\) can have a square root in an order \(R\subseteq S\subseteq \Q R\), while not having such a square root in \(R\).

Let \(R=\Z[2\ii]\) and \(\mathfrak{a}=2R\). 
For \(S=\Z[\ii]\) and \(\mathfrak{c}=(1+\ii)S\) we have \(\mathfrak{c}^2=2\ii S= \mathfrak{a}S\), so \(\mathfrak{a}\) has a square root in a larger order.
Since \(S\) is Dedekind, the group \(\mathcal{I}(S)\) is torsion-free, so \(\mathfrak{c}\) is even the unique square root of \(\mathfrak{a}S\).
Suppose \(\mathfrak{b}\) is some fractional ideal of \(R\) such that \(\mathfrak{b}^2=\mathfrak{a}\).
Then \(\mathfrak{b}S=\mathfrak{c}\) by uniqueness of \(\mathfrak{c}\), so \(\mathfrak{b}\subseteq\mathfrak{c}\subseteq S\).
Let \(x\in\mathfrak{b}\). Then \(x=s+t\ii\) for \(s,t\in\Z\).
As \(2R = \mathfrak{b}^2 \ni x^2 = (s^2-t^2)+2st\ii\), we conclude that \(s,t\in2\Z\). Hence \(\mathfrak{b}\subseteq 2 S\). But then \(2R=\mathfrak{b}^2\subseteq 4 S\), which is false.
Hence \(\mathfrak{b}\) does not exist.
\end{example}

\begin{example}\label{ex:non-functor}
There is no functor that takes as input a square fractional ideal in some order and outputs a square root of this ideal (in the same order).

Consider \(R=\Z[2\textup{i}]\) with invertible fractional ideals \(\mathfrak{b}=2(1+\textup{i}) R\subseteq R\) and \(\mathfrak{c}=2(1-\textup{i})R\subseteq R\).
We have \(\mathfrak{b}^2=8\textup{i} R =\mathfrak{c}^2\).
Note that \(8\textup{i} R\) is invariant under the automorphism group of \(R\), so likewise should a functorially chosen square root of it be invariant.
Since \(\mathfrak{b}\) and \(\mathfrak{c}\) are distinct conjugates, there should be a third square root of \(8\textup{i}R\).
We will show that the \(2\)-torsion subgroup \(\mathcal{I}(R)[2]\) of \(\mathcal{I}(R)\) has cardinality \(2\), giving a contradiction.

Suppose \(\mathfrak{a}\in\mathcal{I}(R)\) satisfies \(\mathfrak{a}^2=R\). Write \(S=\Z[\textup{i}]\) for the maximal order.
Then \((S\mathfrak{a})^2=S\), and because \(S\) is Dedekind also \(S\mathfrak{a} = S\), so \(\mathfrak{a}\subseteq S\).
On the other hand we have \(\mathfrak{a} = \mathfrak{a}^2 (R : \mathfrak{a}) \supseteq R (R : S) = 2S\).
Hence \(\mathfrak{a}\) corresponds to some subgroup of \(S/2S\).
Clearly \(\mathfrak{a}\) is neither \(S\) nor \(2S\), leaving \(3\) possible subgroups.
However, the order of \(\mathcal{I}(R)[2]\) is a non-trivial power of \(2\), so this power must be \(2\), as was to be shown.
\end{example}

If \(R\) is an order such that the roots of all its fractional ideals are unique, then in particular \(\mathcal{I}(R)\) is torsion-free. Proposition~\ref{prop:torsion_free_compute_OK} will show that we cannot expect to compute an order with the latter property in polynomial time. 

\begin{lemma}
If \(R\subseteq S\) are Artinian commutative (sub)rings such that \(R^*=S^*\), then \(2S\subseteq R\).
\end{lemma}
\begin{proof}
Firstly, suppose \(R\) is a field.
If \(R\) has characteristic \(2\), then it is clear we are done.
Otherwise $\#(S/\m)>2$ for every maximal ideal $\m$, and every element of \(S\) is a sum of two units: 
write $S=\prod_{\m} S_\m$ as a finite product of local rings and observe that each $(x_\m)_\m\in S$ equals $(x_\m-y_\m)_\m+(y_\m)_\m$ for any $(y_\m)_\m\in S$ such that $y_\m\not\equiv 0,x_\m\ (\textup{mod }\m)$.
Thus \(R^*=S^*\) implies \(R=S\).

Secondly, suppose that \(R\) is local with maximal ideal \(\p\).
Let \(\mathfrak{n}\) be the nilradical of \(S\). Since \(1+\mathfrak{n}\subseteq S^*=R^*\), we have that \(\mathfrak{n}\subseteq R\). Hence \(\mathfrak{n}=\p\), the nilradical of \(R\). Then \(R/\mathfrak{n}\) is a subring of \(S/\mathfrak{n}\), to which we apply the previous case. Hence \(S/R\cong(S/\mathfrak{n})/(R/\mathfrak{n})\) has exponent \(2\) as Abelian groups.

Thirdly, consider the general case. Then \(R\cong\prod_i R_i\) is a finite product of local rings, and \(S\cong \prod_i S_i\) factors accordingly. It is easy to see that \(R_i^*=S_i^*\) for all \(i\), so we reduce to the previous case.
\end{proof}

\begin{proposition}\label{prop:torsion_free_compute_OK}
There exists a polynomial-time algorithm that, given an order \(R\) of maximal rank in a number field \(K\) such that \(\mathcal{I}(R)\) is torsion-free, computes \(\mathcal{O}_K\).
\end{proposition}
\begin{proof}
We claim \(2 S\subseteq R\), where \(S=\mathcal{O}_K\).
Write \(\mathfrak{f}=R:S\) and note that \(\mathfrak{f}\) is an ideal of both \(R\) and \(S\).
Hence \(R/\mathfrak{f}\) is a subring of \(S/\mathfrak{f}\), and it suffices by the previous lemma to show that the natural map \((R/\mathfrak{f})^*\to(S/\mathfrak{f})^*\) is surjective.
Let \(x,y\in S\) be such that \(xy\equiv 1\ (\textup{mod }\mathfrak{f})\).
As \(x\mathfrak{f},y\mathfrak{f}\subseteq R\), we have \((Rx+\mathfrak{f})(Ry+\mathfrak{f})=R\), so \(Rx+\mathfrak{f}\in\mathcal{I}(R)\).
As \(S(Rx+\mathfrak{f})=Sx+\mathfrak{f}=S\), we conclude from Lemma~\ref{lem:torsion_equiv} that \(Rx+\mathfrak{f}\) is torsion, so by assumption \(Rx+\mathfrak{f}=R\).
Hence \(x+\mathfrak{f}\in (R/\mathfrak{f})^*\), proving the claim.

From the claim it follows that the singular primes of \(R\) lie above \(2\). 
We may simply compute the primes above $2$ in polynomial time using \cite{Berlekamp} and iteratively compute their blowup using Theorem~\ref{thm:blowup} until all primes are invertible, and hence the resulting order is maximal.
\end{proof}

\bibliographystyle{plain} 
\bibliography{citations}

\end{document}